\documentclass[a4paper,12pt]{article}

\usepackage[utf8x]{inputenc}
\usepackage{amsmath}
\usepackage{amsfonts}
\usepackage{latexsym}
\usepackage{amsthm}
\usepackage{amssymb,amscd}
\usepackage{xargs}
\usepackage{graphicx}
\usepackage{color}
\usepackage{enumitem}
\newtheorem{thm}{Theorem}[section]

\newtheorem{lem}[thm]{Lemma}

\theoremstyle{definition}
\newtheorem{cor}[thm]{Corollary}
\newtheorem{rmk}[thm]{Remark}
\newcommand{\be}{\begin{eqnarray}}
\newcommand{\ee}{\end{eqnarray}}
\newcommand{\beal}{\begin{aligned}}
\newcommand{\enal}{\end{aligned}}

\newcommand{\eps}{\varepsilon}
\newcommand{\tet}{\theta}

\newcommand{\de}{\delta}
\newcommand{\bt}{\beta}
\newcommand{\wt}{\widetilde}
\newcommand{\rr}{\rho}

\newcommand{\kk}{\kappa}
\newcommand{\sg}{\sigma}
\newcommand{\ga}{\gamma}
\newcommand{\lb}{\lambda}
\newcommand{\E}{\mathbb{E}}
\newcommand{\Ev}{\mathbb{E}v}
\newcommand{\Eu}{\mathbb{E}u}
\newcommand{\Ef}{\mathbb{E}f}
\newcommand{\Prob}{\mathbb{P}}
\newcommand{\T}{\mathbb{T}}
\newcommand{\R}{\mathbb{R}}
\newcommand{\A}{\mathbb{A}}
\newcommand{\Z}{\mathbb{Z}}
\newcommand{\ZZ}{\mathbb{Z}}
\newcommand{\Lb}{\Lambda}
\newcommand{\om}{\omega}

\newcommand{\DD}{\mathcal{D}}

\newcommand{\TT}{\mathbb{T}}

\newcommand{\OO}{\mathcal{O}}
\newcommand{\I}{\mathcal{I}}
\newcommand{\CCC}{\mathcal{C}}
\newcommand{\pa}{\partial}
\newcommand{\cR}{\mathcal{R}}

	\def\textb{\textcolor{blue}}

\title{Random Iteration of Cylinder Maps and diffusive behavior away from
resonances}
\author{
O. Castej\'on\footnote{Universitat Polit\`ecnica de Catalunya,
oriol.castejon@gmail.com}, M. Guardia\footnote{Universitat Polit\`ecnica de
Catalunya and Barcelona Graduate School of Mathematics,
marcel.guardia@upc.edu}\ \ and \  
V. Kaloshin\footnote{University of Maryland at College Park,
vadim.kaloshin@gmail.com}}

\begin{document}
\maketitle

\begin{abstract}
In this paper we propose a model of random compositions of 
cylinder maps, which in the simplified form is as follows: let 
$(\theta,r)\in \mathbb T\times \mathbb R=\mathbb A$ and 
\begin{eqnarray} \label{model}
f_{\pm 1}:
\left(\begin{array}{c}\theta\\r\end{array}\right) & 
\longmapsto &  
\left(\begin{array}{c}\theta+r+\varepsilon u_{\pm 1}(\theta,r)
\\
r+\varepsilon v_{\pm 1}(\theta,r)
\end{array}\right),
\end{eqnarray} 
where $u_\pm$ and $v_\pm$ are smooth and $v_\pm$ 
are trigonometric polynomials in $\theta$ such that 
$\int v_\pm(\theta,r)\,d\theta=0$ for each $r$. 
We study the random compositions  
$$
(\theta_n,r_n)=f_{\om_{n-1}}\circ \dots \circ f_{\om_0}(\theta_0,r_0),
$$
where $\om_k =\pm 1$ with equal probability. 
We show that under  non-de\-generacy hypotheses and away from resonances
for $n\sim \eps^{-2}$ the distributions of $r_n-r_0$ weakly 
converge to a stochastic diffusion process with explicitly 
computable drift and variance. 

{In the case 
$u_\pm(\theta)=v_\pm(\theta)$ are trigonometric 
polynomials of zero average 
we prove {\it a vertical central limit theorem,} 
namely, for $n\sim \eps^{-2}$ the distributions 
of $r_n-r_0$ weakly converge to the normal 
distribution $\mathcal N(0,\sigma^2)$ with 
$\sigma^2=\frac14\int (v_+(\theta)-v_-(\theta))^2\,d\theta$.} 

{The random model (\ref{model}) up to higher order 
terms in $\eps$ is conjugate to a restrictions to a Normally 
Hyperbolic Invariant Lamination of the generalized Arnold 
example (see \cite{GKZ,KZZ}). Combining the result of 
this paper with \cite{CGK,GKZ,KZZ} we show formation of 
stochastic diffusive behaviour for the generalized Arnold example. }
\end{abstract}

\tableofcontents

\section{Introduction}

\subsection{Motivation: Arnold diffusion and instabilities}

By Arnold-Liouville theorem a completely integrable Hamiltonian system 
can be written in action-angle coordinates, namely, for action 
$p$ in an open set $U\subset \R^n$ and angle $\theta$ on 
an $n$-dimensional torus $\T^n$ there  is a function $H_0(p)$ 
such that equations of motion have the form 
\[
\dot \theta=\om(p),\quad \dot p=0, \qquad \text{ where }\ \om(p):=\partial_p H_0(p).
\]
The phase space is foliated by invariant $n$-dimensional tori $\{p=p_0\}$ 
with either periodic or quasi-periodic motions 
$\theta(t)=\theta_0+t\,\om (p_0)$ (mod 1). There are many different 
examples of integrable systems (see e.g. wikipedia).

It is natural to consider small Hamiltonian perturbations 
\[
H_\eps (\theta,p)=H_0(p)+\eps H_1(\theta,p),\qquad \theta\in\T^n,\ p\in U
\]
where $\eps$ is small.  The new equations of motion become  
\[
\dot \theta=\om(p)+\eps \partial_pH_1,\quad \dot p=-\eps \partial_\theta H_1. \qquad \qquad 
\]

In the sixties, Arnold \cite{Arn64} (see also \cite{Arn89, Arn94}) 
conjectured that {\it for a generic analytic perturbation there 
are orbits $(\theta,p)(t)$ for which the variation of the actions 
is of order one, i.e. $\|p(t)-p(0)\|$ that is bounded from 
below independently of $\eps$ for all $\eps$ sufficiently small.}

See \cite{BKZ,Ch,KZ12,KZ14a, Ma2, Ma3} about recent 
progress proving this conjecture for convex Hamiltonians. 

\subsection{KAM stability}

Obstructions to any form of instability, in general, and 
to Arnold diffusion, in particular,  are widely known, 
following the works of Kolmogorov, Arnold, and Moser,  nowadays called KAM 
theory. The fundamental result 
says that for a properly non-degenerate $H_0$ and for all 
sufficiently regular perturbations $\eps H_1$, the system 
defined by $H_\eps$ still has many invariant $n$-dimensional 
tori. These tori are small deformation 
of unperturbed tori and measure of the union of these invariant 
tori  tends to the full measure as $\eps$ goes to zero. 

One consequence of KAM theory is that for $n=2$ there are 
no instabilities. Indeed, generic energy surfaces $S_E=\{H_\eps =E\}$ are 
$3$-dimensional manifolds whereas KAM tori are $2$-dimensional. 
Thus, KAM tori separate surfaces $S_E$ and prevent orbits from diffusing. 

\subsection{A priori unstable systems}
In \cite{Arn64} Arnold proposed to study 
the following important  example 
\[
\beal 
H_\eps (p,q,I,\varphi,t)=\dfrac{I^2}{2}+H_0(p,q)
+\eps H_1(p,q,I,\varphi,t):=
\qquad \qquad  \qquad \qquad 
\\
=\underbrace{\dfrac{\ \ \ \ \ \ \ I^2\ \ \ \ \ \ \ }{2}}_{harmonic\ oscillator}+
\underbrace{\dfrac{p^2}{2}+(\cos q-1)}_{pendulum}+
\eps H_1 (p,q,I,\varphi,t),
\enal
\]
where $q,\varphi,t\in \T$ --- angles, $p,I\in \R$ --- actions 
(see Figure \ref{fig:rotor-pendulum}), $H_1=(\cos q-1)(\cos \varphi+\cos t)$.

\begin{figure}[h]
  \begin{center}
  \includegraphics[width=10cm]{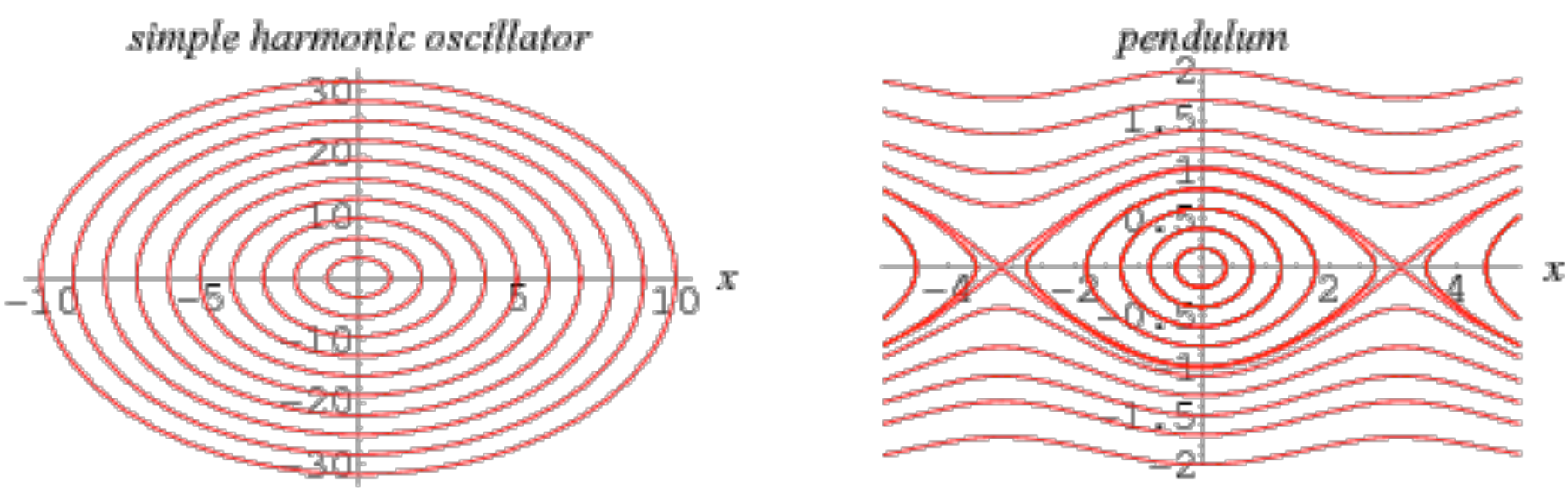}
  \end{center}
  \caption{The rotor times the pendulum }
  \label{fig:rotor-pendulum}
\end{figure}

For $\eps=0$ the system is a direct product of the harmonic 
oscillator  $\ddot \varphi=0$ and the pendulum $\ddot q=\sin q$. 
Instabilities occur when the $(p,q)$-component follows 
the separatrices $H_0(p,q)=0$ and passes near the saddle
$(p,q)=(0,0)$. Equations of motion for $H_\eps$ have 
a (normally hyperbolic) invariant cylinder $\Lb_\eps$ 
which is $\mathcal{C}^1$ close to $\Lb_0=\{p=q=0\}$.  
Systems having an invariant cylinder with a family of 
separatrix loops are called {\it a priori unstable}. 
Since they were introduced by Arnold \cite{Arn64}, they 
received a lot of attention both in mathematics and physics 
community see e.g. \cite{Be,CY,Ch,CV,DLS,GL,T2,T3}. 

Chirikov \cite{Ch2} and his followers made extensive numerical
studies for the Arnold example. He conjectured that 
{\it the $I$-displacement  behaves  randomly, 
where randomness is due to choice of initial conditions
near $H_0(p,q)=0$}.

More exactly, integration of solutions whose ``initial conditions''
randomly chosen $\eps$-close to $H_0(p,q)=0$ and integrated 
over time $\sim \eps^{-2}\ln \eps^{-1}$\,-time. This leads to 
the $I$--displacement being of order of one and having some 
distribution. This coined the name for this phenomenon:
{\it Arnold diffusion}.

Let $\eps=0.01$ and $T=\eps^{-2}\ln \eps^{-1}$. 
On Fig. \ref{fig:histograms} we present several 
histograms plotting displacement of the $I$-component 
after time $T, 2T, 4T,8T$ with 6 different groups of initial 
conditions,  and histograms of $10^6$ points. In each 
group we start with a large set of initial conditions close 
to $p=q=0,\ I=I^*$.\footnote{These histograms are part of 
the forthcoming paper of the third author with P. Roldan 
with extensive numerical analysis of dynamics of the Arnold's 
example.} One of the distinct features is that only one distribution 
(a) is close to symmetric, while in all others have a drift. 

\begin{figure}[h!]
  \begin{center}
  \includegraphics[width=11.55cm]{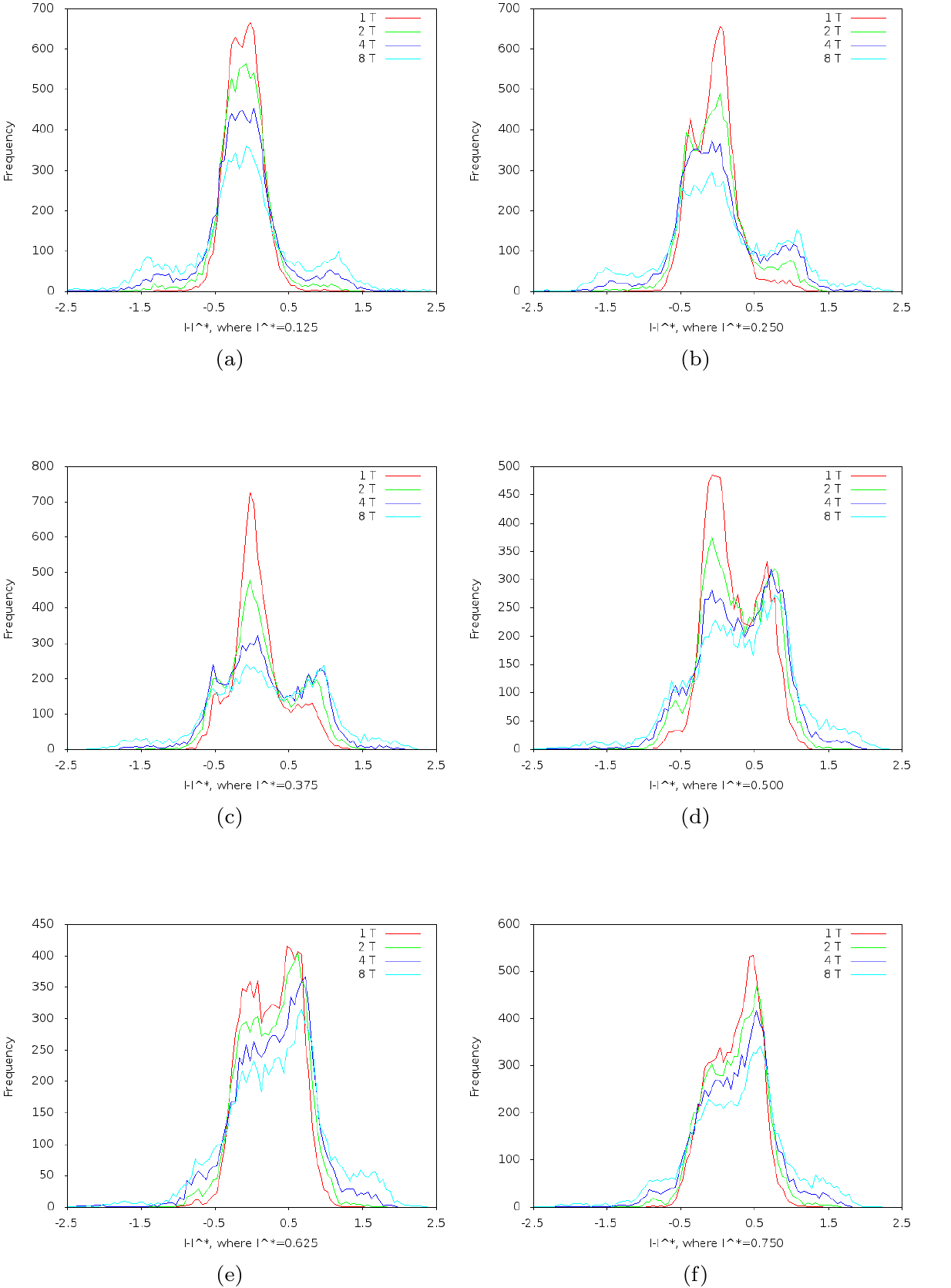}
  \end{center}
  \label{fig:histograms}
\end{figure}

A similar stochastic behaviour was observed numerically in 
many other nearly integrable problems (\cite{Ch2} pg. 370, 
\cite{DL, La}, see also \cite{SLSZ}). To give another illustrative 
example consider motion of asteroids in the asteroid belt.

\subsection{Fluctuations of eccentricity in 
Kirkwood gaps in the asteroid belt}
The asteroid belt is located between orbits of Mars and 
Jupiter and has around one million  asteroids of diameter 
of at least one kilometer. When astronomers build 
a histogram based on orbital period of asteroids there are 
well known gaps in distribution called {\it Kirkwood gaps}
(see Figure below).

\begin{figure}[h!]
  \begin{center}
  \includegraphics[width=8cm]{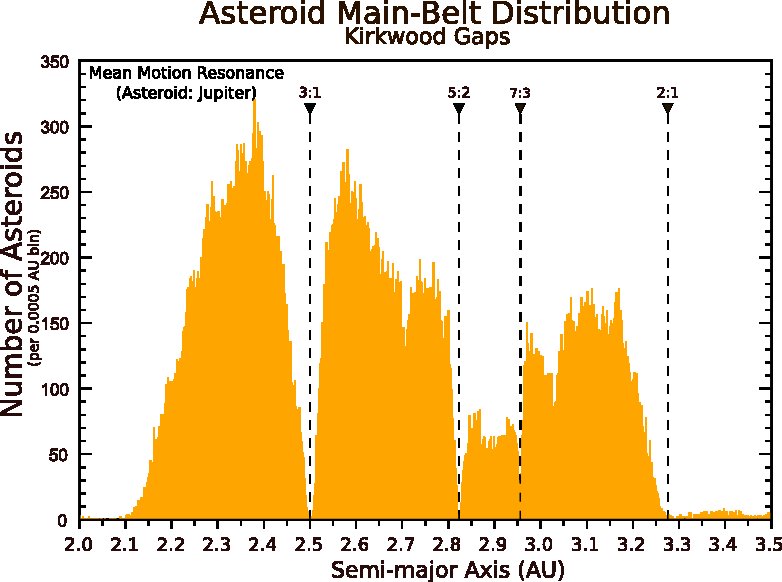}
  \end{center}
  \label{fig:Kirkwood-gaps}
\end{figure}
These gaps occur when the ratio of periods of an asteroid and  
Jupiter is a rational with small denominator: $1/3,2/5,3/7,1/2$.  
This corresponds to so called {\it mean motion resonances
for the three body problem}. 

{Wisdom \cite{Wi} made a numerical analysis of dynamics at 
the  $1/3$ resonance and observed drastic jumps of eccentricity 
of asteroids, which are large enough so that an orbit of asteroid 
starts crossing the orbit of Mars. Once orbits do cross, they eventually undergo
 ejection, or collision, or capture. 
Later it was shown that this mechanism of jumps applies to 
the $2/5$ resonance. However, resonances $3/7$ and $1/2$ 
exhibited a different nature of instability (see e.g. \cite{Moo}). }

{In \cite{FGKR} for small (unrealistic) eccentricity of 
Jupiter, we construct
a dynamical structure along  the $1/3$ resonance which hypothetically 
leads to random 
fluctuations of eccentricity. Using this structure we prove 
existence of orbits whose eccentricity change by $\OO(1)$ for 
the restricted planar three body problem.}

Outside of these resonances one could argue that KAM 
theory provides stability see e.g. \cite{Mo}.

\subsection{Random iteration of cylinder maps}
Consider the time one map of $H_\eps$, denoted 
$$
F_\eps:(p,q,I,\varphi)\to (p',q',I',\varphi').
$$ 
It turns out that for initial conditions in certain domains $\eps$-close 
to $H_0(p,q)=0$, one can define 
a return map to an $\OO(\eps)$-neighborhood of $(p,q)=0$. 
Often such a map is called {\it a separatrix map} and in 
the $2$-dimensional case was introduced by the physicists 
Filonenko-Zaslavskii \cite{FZ}. 
In multidimensional setting such a map was defined and 
studied by Treschev \cite{PT,T1,T2,T3}. 

It turns out that starting near $(p,q)=0$ and iterating $F_\eps$ until 
the orbit comes back $(p,q)=0$ leads to a family of maps 
of a cylinder 
$$
f_{\eps,p,q}:(I,\varphi) \to (I',\varphi'), \qquad 
(I,\varphi)\in \A=\R\times \T
$$ 
which are close to integrable. Since at $(p,q)=0$ 
the $(p,q)$-component has a saddle, there is a sensitive 
dependence on initial condition in $(p,q)$ and returns do 
have some randomness in $(p,q)$. The precise nature 
of this randomness at the moment is not clear. There are 
several coexisting behaviours, including unstable diffusive,
stable quasi-periodic, orbits can stick to KAM tori. Which
behavior is dominant is yet to be understood. May be also the
mechanism of capture into resonances \cite{Do}
is also relevant in this setting. 

In \cite{KZZ} we construct a normally hyperbolic invariant 
lamination (NHIL) for an open class of trigonometric 
perturbations 
$H_1=P(\exp(i\varphi),\,\exp(i t),\,\exp(iq)).$ 

Constructing 
unstable orbits along a NHIL is also discussed in 
\cite{dlL}. In general, NHILs give rise to a skew shift. 
For example, let $\Sigma=\{-1,1\}^\Z$ be the space of 
infinite sequences of $-1$'s and $1$'s and 
$\sigma:\Sigma \to \Sigma$ be the standard shift. 

\vskip 0.1in 

{\it Consider a skew product of cylinder maps 
$$F:\mathbb A \times \Sigma \to \mathbb A \times \Sigma,
\qquad F(r,\theta;\om)=(f_\om(r,\theta),\sigma \om),  
$$
where each $f_\om(r,\theta)$ is a nearly integrable cylinder 
maps, in the sense that it almost preserves the $r$-component
\footnote{The reason we switch 
from the $(I,\varphi)$-coordinates on  the cylinder to 
$(r,\theta)$ is because we perform a coordinate change.}.} 

\vskip 0.1in 

The goal of the present paper is to study a wide enough 
class of skew products so that they arise in Arnold's example
with a trigonometric perturbation of the above type 
(see \cite{GKZ,KZZ}).
 
Now we formalize our model and present the main result.

\subsection{Diffusion processes and infinitesimal generators}
\label{sec:diffusion-generators}

We recall some basic probabilistic notions. 
Consider a Brownian motion
 $\{B_t,\, t\ge 0\}$. 

It is a properly chosen limit of the standard 
random walk. A generalisation of a Brownian motion is 
{\it a diffusion process} or {\it an Ito diffusion}. To define it 
let $(\Omega,\Sigma,P)$ be a probability space. 
Let $R:[0,+\infty) \times \Omega \to \R$. It is called an Ito diffusion 
if it satisfies {\it a stochastic differential equation} of the form
\begin{equation}\label{eq:diffusion}
\mathrm{d} R_{t} = b(R_{t}) \, \mathrm{d} t + 
\sigma (R_{t}) \, \mathrm{d} B_{t},
\end{equation}
where $B_t$ is a Brownian motion and $b : \R \to \R$ and 
$\sigma : \R \to \R$ are Lipschitz functions called 
the drift and the variance respectively. For a point 
$r \in \R$, let $\mathbb{P}_r$ denote the law of $X$ 
given initial data $R_0 = r$, and let $\E_r$ denote 
expectation with respect to $\mathbb{P}_r$.

The {\it infinitesimal generator} of $R$ is the operator $A$, 
which is defined to act  on suitable functions $f :\R\to \R$ by
\[
A f (r) = \lim_{t \downarrow 0} \dfrac{\E_{r} [f(R_{t})] - f(r)}{t}.
\]
The set of all functions $f$ for which this limit exists at 
a point $r$ is denoted $D_A(r)$, while $D_A$ denotes 
the set of all $f$'s for which 
the limit exists for all $r\in \R$. One can show that any 
compactly-supported $\mathcal{C}^2$  function $f$ 
lies in $D_A$ and that
\be \label{eq:diffusion-generator}
Af(r)=b(r) \dfrac{\partial f}{\partial r}+ \dfrac 12 \sigma^2(r)
\dfrac{\partial^2 f}{\partial r \partial r}.
\ee
The distribution of a diffusion process is characterized  by 
the drift $b(r)$ and the variance $\sigma(r)$.

\section{The model and statement of the main result}
Let $\eps>0$ be a small parameter and $l\ge 7$, 
$s\geq 0$ be integers. Denote by $\mathcal O_s(\eps)$  
a $\mathcal C^s$ function whose $\mathcal C^s$ norm is 
bounded by $C\eps$ with $C$ independent of $\eps$. 
Similar definition applies for a power of $\eps$. As before 
$\Sigma$ denotes $\{-1,1\}^\Z$ and 
$\om=(\dots,\om_0,\dots)\in \Sigma$. 

Consider nearly integrable maps
\begin{eqnarray} \label{mapthetar}
f_\om:\mathbb{T}
\times \mathbb{R}
& \longrightarrow & 
\mathbb{T} \times \mathbb{R} \qquad \qquad 
\qquad \qquad \qquad \qquad
\nonumber\\
f_\om:
\left(\begin{array}{c}\theta\\r\end{array}\right) & 
\longmapsto &  
\left(\begin{array}{c}\theta+r+\eps u_{\om_0}(\theta,r)+
\mathcal O_s(\eps^{1+a},\om)
\\
r+\eps v_{\om_0}(\theta,r)+
\eps^2 w_{\om_0}(\theta,r)+\mathcal O_s(\eps^{2+a},\om)
\end{array}\right),
\end{eqnarray} 
for $\om_0\in \{-1,1\}$, where $u_{\om_0},\ v_{\om_0},$
and $w_{\om_0}$ are bounded $\mathcal{C}^l$ functions, 
$1$-periodic in $\theta$, $\mathcal O_s(\eps^{1+a},\om)$ 
and $\mathcal O_s(\eps^{2+a},\om)$ denote  remainders 
depending on $\om$ and uniformly $\mathcal C^s$ bounded 
in $\om$, and $a> 1/2$. Assume 
\[
\max |v_i(\theta,r)|\le 1,
\]
where maximum is over $i=\pm 1$ and all 
$(\theta,r)\in \A$, otherwise, renormalize $\eps$, and 
\[
 \|u_i\|_{\CCC^6},  \|v_i\|_{\CCC^6},  \|w_i\|_{\CCC^6}\leq C
\]
for some $C>0$ independent of $\eps$.

Even if the maps $f_\omega$ depend on the full sequence $\omega$, the dependence
on the elements of $\omega_k$, $k\neq 0$, is rather weak since only appear in
the small remainder. Therefore, we abuse notation and we denote these maps
as $f_1$ and $f_{-1}$. Certainly we do not have two but an infinite number of
maps. Nevertheless, they can be treated as just two maps since the remainders
are negligible.

We study the random iterations of these maps $f_1$ and $f_{-1}$, assuming that 
at each step the probability of performing either 
map is $1/2$. The importance of understanding iterations of 
several maps for problems of diffusion is well known
(see e.g. \cite{K,Mo}). 

Denote the expected potential and 
the difference of potentials by 
\[
\begin{split}
\Eu(\theta,r)&:=
\frac 12 (u_1(\theta,r)+u_{-1}(\theta,r)),\ \ \  \Ev(\theta,r):=\frac 12
(v_1(\theta,r)+v_{-1}(\theta,r)),\\
u(\theta,r)&:=\frac 12 (u_1(\theta,r)-u_{-1}(\theta,r)),\ \ \ 
v(\theta,r):=\frac 12 (v_1(\theta,r)-v_{-1}(\theta,r)).
\end{split}
\]
Suppose the following assumptions hold:
\begin{itemize}
\item[{\bf [H0]}] ({\it zero average})
{For} each $r\in \R$ and $i=\pm 1$ we have  
$\int v_i(\theta,r)\,d\theta=0$.

\item[{\bf [H1]}] 
for each $r\in \R$ we have $\int_0^1\ v^2(\theta,r)d\theta=:\sigma(r) \neq0$;

\item[{\bf [H2]}] The functions $v_i(\theta,r)$ are trigonometric polynomials 
in $\theta$, i.e. for some positive integer $d$ we have 
\[
v_i(\theta,r)=\sum_{k\in \Z,\ 0<|k|\le d} 
v_{{i}}^{(k)}(r) e^{2\pi ik\theta}. 
\]

\item[{\bf [H3]}] ({\it no common zeroes}) 
For each integer $n\in \Z$ potentials $v_{1}(\theta,n)$ and 
$v_{-1}(\theta,n)$ have no common zeroes and, equivalently, 
$f_1$ and $f_{-1}$ have no fixed points.

\item[{\bf [H4]}]  ({\it no common periodic orbits}) 
Take any rational $r=p/q\in\mathbb Q$ 
with $p,q$ relatively prime, $1\le |q|\le 2d$ and any $\theta^*\in \T$  such
that for all $\theta$ either 
\[
\sum_{k=1}^q v_{-1}\left(\theta+\frac kq,\frac{p}{q}\right) \ne 
0
\]
or 
\[
\sum_{k=1}^q\left[v_{-1}\left(\theta+\frac kq,\frac{p}{q}\right)-
v_1\left(\theta+\frac kq,\frac{p}{q}\right)\right]^2\ne 
0.
\]
This prohibits $f_1$ and $f_{-1}$ to have common periodic 
orbits of period $|q|$. 

\item[{\bf [H5]}] ({\it no degenerate periodic points}) 
Suppose for any rational $r=p/q\in\mathbb Q$ 
with $p,q$ relatively prime, $1\le |q|\le 2d$, the function:
$$\Ev_{p,q}(\theta,r)=\sum_{\substack{k\in 
\mathbb{Z}\\0<|kq|<d}}\Ev^{kq}(r)e^{2\pi ikq\theta}$$
has distinct non-degenerate zeroes, where $\Ev^{j}(r)$ 
denotes the $j$--th Fourier coefficient of $\Ev(\theta,r)$.
\end{itemize}


 

For $\omega\in\{-1,1\}^\Z$ we 
can rewrite the maps $f_{\om}$ in the following form:
\begin{equation*}
f_{\omega}
\left(\begin{array}{c}\theta\\r\end{array}\right)\longmapsto
\left(\begin{array}{c}\theta+r+\eps \Eu(\theta,r)+ 
\eps\omega_0 u(\theta,r)+\mathcal O_s(\eps^{1+a},\om)
\\
r+\eps \Ev(\theta,r)
+\eps\omega_0 v(\theta,r)+\eps^2 w_{\om_0}(\theta,r)
+\mathcal O_s(\eps^{2+a},\om)
\end{array}\right).
\end{equation*}

Let $n$ be a positive integer and $\omega_k\in\{-1,1\}$, $k=0,\dots,n-1$, be 
independent random variables with $\mathbb{P}\{\omega_k=\pm1\}=1/2$ and 
$\Omega_n=\{\omega_0,\dots,\omega_{n-1}\}$. 
Given an initial condition $(\theta_0,r_0)$ we denote
\[
(\theta_n,r_n):=f^n_{\Omega_n}(\theta_0,r_0)=
f_{\omega_{n-1}}\circ f_{\omega_{n-2}}\circ \cdots
\circ f_{\omega_0}(\theta_0,r_0).
\]
%
%

A straightforward calculation shows that:
\begin{equation}\label{mapthetanrn}
\begin{array}{rcl}
\theta_n&=&\displaystyle\theta_0+nr_0+\eps\left( 
\sum_{k=0}^{n-1}
\Eu(\theta_k,r_k)+ \sum_{k=0}^{n-2}(n-k-1) \Ev(\theta_k,r_k)\right)
\\
&&\displaystyle+\eps\left(
\sum_{k=0}^{n-1} \omega_k
u(\theta_k,r_k)+\sum_{k=0}^{n-2}(n-k-1) \omega_k v(\theta_k,r_k)\right) 
+\mathcal O_s(n\eps^{1+a})
\bigskip\\
r_n&=&\displaystyle r_0+\eps\sum_{k=0}^{n-1}
\Ev(\theta_k,r_k)+\eps
\sum_{k=0}^{n-1}\omega_kv(\theta_k,r_k)
+\mathcal O_s(n\eps^{2+a})
\end{array}
\end{equation}

\begin{thm}\label{maintheorem}
Assume that, in the notations above, conditions {\bf [H0-H5]} 
hold and take $r_0\in\R$. Let $n_\eps \eps^2 \to s>0$ as $\eps\to 0$ for 
some 
$s>0$. Then as $\eps \to 0$ the distribution of $ r_{n_\eps}-r_0$ 
converges weakly to $R_s$, where $R_\bullet$ is 
a diffusion process of the form \eqref{eq:diffusion}, 
with the drift and the variance 
\be \label{eq:drift-variance}
b(R )=\int_0^1E_2(\theta,R)\,d\theta,
\qquad \sigma^2(R)=\int_0^1v^2(\theta,R)\,d\theta. 
\ee
for certain function $E_2$, defined in \eqref{eq:drift}.
\end{thm}

\paragraph{Remarks}
\begin{itemize}
\item If the map is area preserving and exact, one can check that 
\[
 b(R)=0
\]
(see Corollary \ref{corZeroDrift}).
\item In the case that $u_{\pm1}=v_{\pm 1}$ and that they are 
independent of $r$, we have two area-preserving standard 
maps. In this case the assumptions become 
\begin{itemize}
\item{\bf [H0]} $\int v_i(\tet)d\tet=0$ for $i=\pm 1$;
\item{\bf [H1]} $v$ is not identically zero;
\item{\bf [H2]} the functions $v_i$ are trigonometric
polynomials.
\end{itemize}
A good example is $u_1(\tet)=v_1(\tet)=\cos 2\pi\tet$ and 
$u_{-1}(\tet)=v_{-1}(\tet)=\sin 2\pi\tet$. In this case 
\[
b(r):=\int_0^1E_2(\tet,r)d\tet\equiv 0,\qquad 
\sigma^2=\int_0^1 v^2(\theta)\,d\theta=\frac{1}{4}
\]
and for $n\le \eps^{-2}$ 
the distribution  $r_n-r_0$ converges to the zero mean 
variance $\eps n^2 \sg^2$ normal distribution, 
denoted $\mathcal N(0,\eps n^2 \sg^2)$. More generally,  
we have the following ``vertical central limit theorem'':

\begin{thm}\label{submaintheorem}
Assume that in the notations above conditions {\bf [H0-H5]} 
hold. Let $n_\eps \eps^2 \to s>0$ as $\eps\to 0$ for some 
$s>0$. Then as $\eps \to 0$ the distribution of $ r_{n_\eps}-r_0$ 
converges weakly to a normal random variable 
$\mathcal{N}(0,s^2 \sigma^2).$
\end{thm}

\item Numerical experiments of Moeckel \cite{Moe1} show 
that no common fixed points and  periodic orbits (see  Hypotheses {\bf [H3]} 
and {\bf [H4]}) is not neccessary to deal with the resonant zones. One 
could probably  {replace it} by a weaker non-degeneracy condition, e.g. that 
the linearization of maps $f_{\pm 1}$ at the common fixed and periodic
points  are different. 

\item In \cite{Sa} Sauzin studies random iterations of the standard 
maps  
$$(\theta,r)\to (\theta+r+\lb \phi(\theta),r+\lb \phi(\theta)),
$$
where $\lb$ is chosen randomly from $\{-1,0,1\}$ and proves
the vertical central limit theorem; 
In \cite{MS,Sa2} Marco-Sauzin present examples of nearly 
integrable systems having a set of initial conditions 
exhibiting the vertical central limit theorem. 

\item In \cite{Ma} Marco derives a sufficient condition for a
skew-shift to be a step skew-shift.

\item The condition [H2] that the functions $v_i$ are
trigonometric polynomials in $\theta$ seems redundant too,
however, removing it leads to considerable technical 
difficulties (see Section \ref{sec:different-strips}). In  short, for 
perturbations  
by a trigonometric polynomial there are finitely many resonant 
zones. This finiteness considerably simplifies the analysis.


\item One can replace $\Sigma=\{-1,1\}^\Z$ with 
$\Sigma_N=\{0,1,\dots,N-1\}^\Z$, consider any finite number of maps of the form 
(\ref{mapthetar}) and a transitive Markov chain with some transition probabilities. 
If conditions {\bf [H0--H5]} are satisfied for the proper averages 
$\Ev$ of $v$, then Theorem \ref{maintheorem} holds. 

%
%
%

\end{itemize}

\section{Strategy of the proof}
The random map \eqref{mapthetanrn} has two significantly different regimes: 
resonant and non-resonant. In this paper we analyze 
\eqref{mapthetanrn} 
\emph{away from resonances}. The resonance setting is analyzed in 
\cite{CGK}. The main result of \cite{CGK} is presented in Section
\ref{sec:ResonantRegime}.

We proceed to define the two regimes. Let 
\begin{equation}\label{def:FourierSupp}
 \mathcal N=\{k\in\Z: (\E u^k,\E v^k)\neq 0\}.
\end{equation}
Fix $\beta>0$. 
Then, the $\beta$-non-resonant domain is defined as
\begin{equation}\label{eq:non-res-domain}
\DD_{\bt}=\left\{r\in \R: \forall q\in \mathcal N, \ 
p\in\ZZ\ 
\text{we have }\ \left|r-\frac{p}{q}\right|\ge 2\bt \right\}. 
\end{equation}
Notice that, by Hypothesis \textbf{H2},  $\DD_{\bt}$ contains the subset 
of $\R$ which excludes the   $2\bt$-neighborhoods of all rational numbers 
$p/q$ with $0<|q|\le 2d$. Analogously, we can define the resonant domains
associated to a  rational $p/q$ with $q\in \mathcal N$ as 
\begin{equation}\label{eq:res-domain}
\cR^{p/q}_{\bt}=\left\{r\in \R: \ \left|r-\frac{p}{q}\right|\le 2\bt \right\}. 
\end{equation}

%

\subsection{Strip decomposition}
Fix $\ga\in (0,1)$. We divide the non-resonant zone of the  cylinder, namely 
$\TT\times \DD_\bt$ (see \eqref{eq:non-res-domain}), in strips 
$\mathbb{T}\times I^j_\ga$, where $I^j_\ga\subset\DD_\bt,\ j\in \Z$, 
are intervals of length $\eps^\ga$. Then we study how the random 
variable $r_n-r_0$ behaves in each strip. More precisely, decompose 
the process $r_n(\om), n\in\Z_+$  into infinitely many time intervals 
defined by stopping times 
\be \label{eq:stopping-time}
0<n_1<n_2<\dots,
\ee 
where 
\begin{itemize} 
\item $r_{n_i}(\om)$ is $\eps$-close to 
the boundary between $I^j_\ga$ and $I^{j+1}_\ga$ for 
some $j\in \Z$ 
\item $r_{n_{i+1}}(\om)$ is $\eps$-close to the other 
boundary of either $I^j_\ga$ or of $I^{j+1}_\ga$ and 
$n_{i+1}>n_i$ is the smallest integer with this property. 
\end{itemize} 
Since $\eps\ll \eps^\ga$, being $\eps$-close to the boundary
of $I^j_\ga$ with a negligible error means jump from $I^j_\ga$ 
to the neighbour interval $I^{j\pm 1}_\ga$. In what follows for brevity 
we drop dependence of $r_n(\om)$'s on $\om$.  For reasons which will be clear
in 
Sections \ref{sec:TI-case} and 
\ref{sec:IR-case}, we consider $\ga\in (4/5,4/5+1/40)$.

In \cite{CGK}, we proceed analogously by partitioning the resonant zones.
Nevertheless, the partition is significantly different.
\subsection{Strips with 
different quantitative behaviour}\label{sec:different-strips}

Fix  
\[
\nu = \frac{1}{4}\quad\text{ and }\quad b>0\quad \text{ such that }\quad
\rho:=\nu-2b>0
\]
Consider 
the $\eps^\ga$-grid in the non-resonant zone $\mathcal D_\beta$ (see
\eqref{eq:non-res-domain}). Denote by $I_\ga$ a segment 
whose end points are in the grid. Since in the present paper we only deal with 
the non-resonant zone, we only need to   distinguish among 
the two following types of strips $I_\ga$ (other types for the resonant zones
are defined in \cite{CGK}).
\begin{itemize}
\item\textbf{The Totally Irrational case:} 
A strip $I_\ga$ is called {\it totally irrational} if
$r \in I_\ga$ and $|r-p/q|<\eps^\nu$, with 
$\gcd(p,q)=1$, then $|q|>\eps^{-b}$. 

In this case, we show that there is a good ``ergodization'' 
and 
\[
\sum_{k=0}^{n-1}\omega_kv\left(\theta_0+k\frac{p}{q}\right)\approx 
\sum_{k=0}^{n-1}\omega_kv\left(\theta_0+kr_0^*\right).
\]
for any 
$r_0^*\in I_\ga\cap(\mathbb{R} \setminus\mathbb{Q})$. These strips cover most
of 
the cylinder and give the dominant contribution to 
the behaviour of $r_n-r_0$. Eventually it will lead to the desired 
weak convergence to a diffusion process (Theorem \ref{maintheorem}).

\item \textbf{The Imaginary Rational (IR) case:} 
A strip $I_\ga$ is called {\it imaginary rational} if 
there exists a rational $p/q$ in an $\eps^\nu$ neighborhood of $I_\ga$ 
with $2d<|q|<\eps^{-b}$.

We call these strips Imaginary Rational, 
since the leading term of the angular dynamics
is a rational rotation, however, the associated averaged system
vanishes due to the fact that $u_i$ and $v_i$ only have $k$-harmonics with 
$|k|\leq d$.

In Appendix \ref{sec:measure-IR-RR},  we show that the imaginary rational 
strips occupy an $\mathcal O(\eps^\rho)$-fraction of the cylinder. 
We can show that orbits spend a small fraction of the total time 
in these strips and global behaviour is determined 
by behaviours in the complement.
\end{itemize}

%
%
%
%

%

\subsection{The Normal Forms}
The first step is to find a normal form, so that 
the deterministic part of map \eqref{mapthetanrn} 
is as simple as possible. It is given  in  
Theorem \ref{thm:normal-form}.
In short, we shall see that the 
 deterministic system in both the TI case and 
the IR case are 
a small perturbation of the  twist map
\begin{equation*}
\left(\begin{array}{c}\theta\\r\end{array}\right)
\longmapsto
\left(\begin{array}{c}\theta+r\\r\end{array}\right).
\end{equation*}
On the contrary, in the resonant zones studied in \cite{CGK}, the deterministic 
system will be close to a pendulum-like system
\begin{equation*}
\left(\begin{array}{c}\theta\\r\end{array}\right)
\longmapsto
\left(\begin{array}{c}\theta+r\\r+\eps E(\theta,r)\end{array}\right),
\end{equation*}
for an ``averaged'' potential $E(\theta,r)$ (see Theorem 
\ref{thm:normal-form}, \eqref{NFnear}).
We note that this system has the following approximate first integral
\[
H(\theta,r)=\frac{r^2}{2}-\eps\int_0^\theta E(s,r)ds,
\]
so that indeed it is close to a pendulum-like system. 
This will lead to different qualitative behaviours when considering 
the random system.

\subsection{Analysis of the Martingale problem in each kind of 
strip}\label{sec:Martingale}
The next step is to study the behaviour of the random system 
respectively in Totally Irrational and  Imaginary Rational strips
(see Sections \ref{sec:TI-case} and \ref{sec:IR-case}). 
More precisely, we use a discrete version of the scheme 
by Freidlin and Wentzell \cite{FW}, giving a sufficient condition 
to have weak convergence to a diffusion process  as $\eps\to0$ 
in terms of the associated Martingale problem. Namely, $R_s$ satisfies a 
diffusion process with drift $b(r)$ and variance $\sigma(r)$ provided that
for any $s>0$,  any time 
$n\le s\eps^{-2}$ and any $(\theta_0,r_0)$ we have that as $\eps\to 0$,
\be\label{eq:suff-condition}
\E\left(f(r_{n})- \eps^2
\sum_{k=0}^{n-1}
\left(b(r_k)f'(r_k)+\frac{\sigma^2(r_k)}{2}f''(r_k)\right)
\right)-f(r_0)\to 0.
\ee
This implies the  main result --- 
Theorem \ref{maintheorem}. 

The proof of \eqref{eq:suff-condition}  is done in two steps. 
First, we describe the local behaviour in each strip and 
then we combine the information. We define Markov times $0=n_0<n_1<n_2 <\dots 
<n_{m-1}<n_m=n\leq s\eps^{-2}$
for some random $m=m(\om)$ such that each $n_k$ is the stopping 
time as in (\ref{eq:stopping-time}) and $n_m$ is the final time. Almost surely
$m(\om)$ is 
finite.  We decompose the above sum 
\[
\E\left(\sum_{k=0}^{m-1} \left[f(r_{n_{k+1}})-
f(r_{n_{k}})-\eps^2
\sum_{s=n_k}^{n_{k+1}}\left(b(r_s)f'(r_s)+\frac{\sigma^2(r_s)}{2}
f''(r_s)\right)\right]\right),
\]
analyze each  summand in the corresponding strip and then prove that the 
whole sum  converges to 0 as $\eps\to 0$.

\subsubsection{A TI Strip}
\label{sec:TI-prelim}
Let the drift and the variance be as \eqref{eq:drift-variance}.
Let $r_0$ be $\eps$-close to the boundary of two totally 
irrational strips and let $n_\ga$ be stopping of hitting 
$\eps$-neighbourhoods of the adjacent boundaries or 
$n_\ga=n\leq s\eps^{-2}$ be the final time. 
In Lemma \ref{lemmaexpectation} 
we prove that for some $\zeta>0$ 
\begin{eqnarray}\label{exp-formula}
\beal 
&&\E\left(f(r_{n_\ga})- \eps^2
\sum_{k=0}^{n_\ga-1}\left(b(r_k)f'(r_k)+\frac{\sigma^2(r_k)}{2}
f''(r_k)\right)
\right)\\
&& \qquad \qquad \qquad \qquad \qquad  \qquad \qquad 
- f(r_0)=\mathcal{O}(\eps^{2\ga+\zeta}),
\enal 
\end{eqnarray}

\subsubsection{An IR Strip}
\label{sec:IR-prelim}
Consider the drift and variance given in \eqref{eq:drift-variance}.
Let $r_0$ be $\eps$-close to the boundary of an imaginary 
rational strip and let $n_\ga$ be stopping of hitting 
$\eps$-neighbourhoods of the adjacent boundaries   or 
$n_\ga=n\leq s\eps^{-2}$ be the final time. 
Fix any $\de>0$ small. In Lemma \ref{lemmaexpectation-IR} we prove that
\begin{eqnarray*}
 &&\E\left(f(r_{n_\ga})- \eps^2
\sum_{k=0}^{n_\ga-1}
\left(b(r_k)f'(r_k)+
\frac{\sigma^2(r_k)}{2}f''(r_k)\right)\right)\\
&&\qquad \qquad \qquad \qquad \qquad  \qquad \qquad 
\qquad \qquad -f(r_0)=\mathcal{O}(\eps^{2\ga-\de}).
\end{eqnarray*}
\subsection{The resonant zones $\mathcal
R_\beta^{p/q}$}\label{sec:ResonantRegime}
The resonant zones $\mathcal R_\beta^{p/q}$ defined in \eqref{eq:res-domain} are studied in \cite{CGK}. 
We summarize here the key steps
(for a more precise statement see Lemma
\ref{lemma:expectlemmaBigStrips} and remark afterwards below).
Fix $p/q$ with $|q|\leq 2 d$ and consider the associated resonant zone $\mathcal
R_\beta^{p/q}$ for some $\beta>0$ independent of $\eps$ ($\beta$ is chosen so
that the 
different resonant regions do not overlap).

In $\mathcal
R_\beta^{p/q}$ we do not analyze the stochastic behavior in $r$ but in a 
different variable. In \cite{CGK} we show, through a normal form, that, 
after a suitable change of coordinates, the 
deterministic map associated to \eqref{mapthetanrn}  has an approximate first 
integral $H$ of the form 
\[
 H^{p/q}(\theta, r)=\frac{r^2}{2}+\eps V^{p/q}(\theta, 
r)+\OO\left(\eps^2\right).
\]
In the resonant zone \eqref{eq:res-domain}, we analyze the process
 $(\theta_{qn},H_n)$ with 
\[H_n:=H^{p/q}\left(\theta_{qn},R_{qn}\right).\] 
 We prove that, 
 $H_n-H_0$ converges weakly to a diffusion process 
 $H_s$ with $s=\eps^{-2}n$. 
  Notice that the limiting process does not take place on a line. 
 In this case it takes place on a graph, similarly as in \cite{FW}. 
 More precisely, consider the level sets of the function 
 $H^{p/q}(\theta,r)$. The critical points of the potential 
 $V^{p/q}(\theta)$ give rise to critical points of the associated 
 Hamiltonian system. Moreover, if the critical point is a local minimum 
 of $V$, it corresponds to a center of the Hamiltonian system, 
 while if it is a local maximum of $V^{p/q}$, it corresponds to 
 a saddle. Now, if for every value 
 $H\in \mathbb{R}$ we identify all the points $(\theta,r)$ in the same 
 connected component of the curve $\{H^{p/q}(\theta,r)=H\}$, 
 we obtain a graph $\Gamma$ (see Figure \ref{fig:potential-graph} 
 for an example). The interior vertices of this graph represent 
 the saddle points of the underlying Hamiltonian system jointly with 
 their separatrices, while the exterior vertices represent the centers 
 of the underlying Hamiltonian system. Finally, the edges of the graph 
 represent the domains that have the separatrices as boundaries. 
 The process $H_n$ {can be viewed as a process on the graph.}
 
 \begin{figure}[h]
   \begin{center} 
   \includegraphics[width=8.75cm]{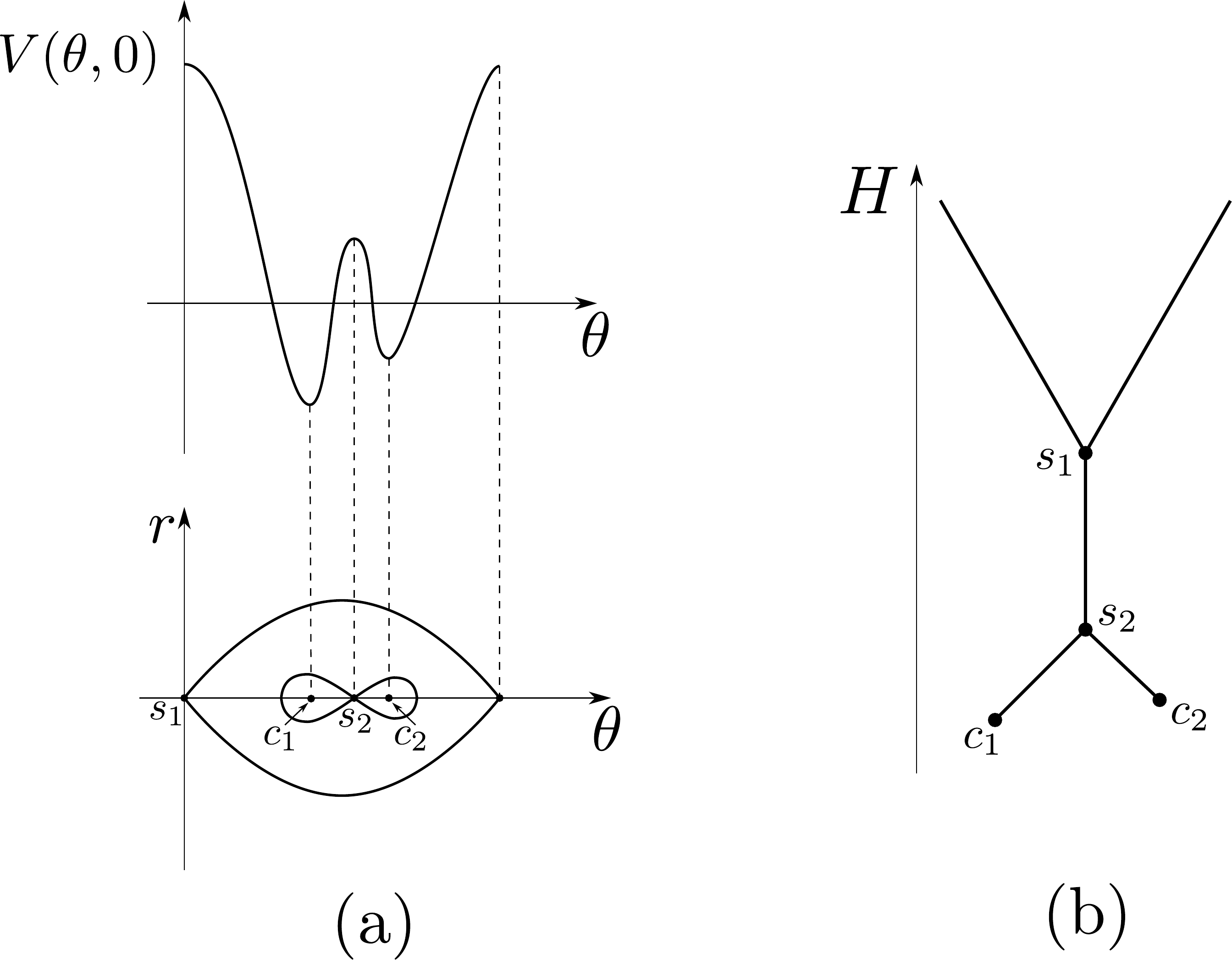}
   \end{center}
   \caption{(a) A potential and the phase portrait of its corresponding 
 Hamiltonian system. (b) The associated graph $\Gamma$.}
   \label{fig:potential-graph}
 \end{figure}

In \cite{CGK} we analyze the stochastic behavior in this graph by proving an 
analogous  sufficient condition to \eqref{eq:suff-condition} on the 
graph. Namely, we use that  $H_s$ satisfies a 
diffusion process provided that
for any $s>0$,  any time 
$n\le s\eps^{-2}$ and any $(\theta_0,H_0)$ we have that as $\eps\to 0$,
\[
\E\left(f(H_{n})- \eps^2
\sum_{k=0}^{n-1}
\left(b(H_k)f'(H_k)+\frac{\sigma^2(H_k)}{2}f''(H_k)\right)
\right)-f(H_0)\to 0, 
\]
then we relate the $H$-process and the $r$-process. 
 
\subsection{Plan of the rest of the paper}
 
In Section \ref{sec:NormalForm} we state and prove 
the normal form theorem for the expected cylinder map 
$\E f$. The main difference with a typical normal form is that 
we need to have not only the leading term in $\eps$,
but also $\eps^2$-terms. The latter terms give information
about the drift $b(r)$ (see (\ref{eq:drift})). 
In Section \ref{sec:TI-case} we analyze the Totally Irrational 
case and prove approximation for the expectation from 
Section \ref{sec:TI-prelim}.  In Section \ref{sec:IR-case} we analyze the
Imaginary Rational case and prove an analogous formula from 
Section \ref{sec:IR-prelim}.  In Section \ref{sec:IR-cyl-to-line} we prove
Theorem \ref{maintheorem} using
the analysis of the TI and IR strips.
%

In Appendix \ref{sec:measure-IR-RR} we estimate measure 
of the complement to the TI strips.  In Appendix \ref{sec:auxiliaries} we
present 
several auxiliary lemmas used in the proofs.

\section{The Normal Form Theorem}\label{sec:NormalForm}
In this section we  prove the Normal Form Theorem, 
which  allows us to deal with the simplest possible 
deterministic system. To this end, we  state 
a technical lemma  needed in the proof of the theorem. 
This is a simplified version (sufficient for our purposes) of Lemma 3.1 
in \cite{BKZ}. 
\begin{lem}\label{lemClnorms}
Let $g(\theta,r)\in\mathcal{C}^l\left(\mathbb{T}\times B\right)$, 
where $B\subset\mathbb{R}$. Then
\begin{enumerate}
	\item If $l_0\leq l$ and $k\neq0$, $\|g_k(r)e^{2\pi i k \theta}\|_{\mathcal{C}^{l_0}}\leq |k|^{l_0-l}\|g\|_{\mathcal{C}^{l}}$.
	\item Let $g_k(r)$ be  functions that satisfy 
	$\|\partial_{ r^\alpha}g_k\|_{\mathcal{C}^0}\leq M|k|^{-\alpha-2}$ 
	for all $\alpha\leq l_0$ and some $M>0$. Then
	$$\left\|\sum_{\substack{k\in\mathbb{Z}\\0<k\leq d}}
	g_k(r)e^{2\pi i k \theta}\right\|_{\mathcal{C}^{l_0}}\leq cM,$$
	for some constant $c$ depending on $l_0$.
\end{enumerate}
\end{lem}

Let $\mathcal{R}$ be the finite set of resonances of the map \eqref{mapthetar}, 
namely, 
\[
\mathcal{R}=\{p/q\in\mathbb{Q}\,:\,\gcd(p,q)=1, |q|\leq 2d\}.
\]



\begin{thm}\label{thm:normal-form}
Consider the expected map $\Ef$ associated to the map \eqref{mapthetar}
\begin{equation}\label{def:ExpectedMap}
\E f
\begin{pmatrix}\theta\\r\end{pmatrix}\longmapsto
\begin{pmatrix}\theta+r+\eps \Eu(\theta,r)
+\mathcal O_s(\eps^{1+a})
\\
r+\eps \Ev(\theta,r)+\eps^2 \E w(\theta,r)+
\mathcal O_s(\eps^{2+a})
\end{pmatrix}.
\end{equation}
Assume that the functions $\Eu(\theta,r)$, $\Ev(\theta,r)$ and $\E w(\theta,r)$ 
are $\mathcal{C}^l$, $l\geq3$. Fix $\beta>0$
small and $0\leq s\leq l-2$. Then, there exists $K>0$ independent of $\eps$ 
and a canonical change of variables
\begin{eqnarray*}
\Phi:\mathbb{T}\times\mathbb{R}&\rightarrow&
\mathbb{T}\times\mathbb{R},\\
(\tilde \theta,\tilde r)&\mapsto&(\theta,r),
\end{eqnarray*}
such that
\begin{itemize}
\item If $|\tilde r-p/q|\geq \beta$ for all $p/q\in\mathcal{R}$, 
then
\begin{equation}\label{NFfar}
\beal 
\Phi^{-1}\circ\Ef\circ\Phi(\tilde\theta,\tilde r)=\qquad 
\qquad \qquad \qquad \qquad \qquad \qquad \\
\begin{pmatrix}\tilde \theta+\tilde r+\eps \Eu(\theta,r)-\eps \Ev(\theta,r)+\eps
E_1(\theta,r)+\mathcal O_s(\eps^{1+a})+\mathcal{O}_s(\eps^2\beta^{-(2s+4)})
\\ 
\tilde r+\eps^2E_2(\tilde\theta,\tilde
r)+\mathcal{O}_s(\eps^{2+a})+\mathcal{O}_s(\eps^3\beta^{-(3s+5)})\end{pmatrix}
,
\enal 
\end{equation}
where $E_1$ and $E_2$ are some $\mathcal{C}^{l-1}$ functions. There exists a 
constant $K$ such that for any $0\leq s\leq l-1$ one has
$$\|E_1\|_{\mathcal{C}^s}\leq K\|\Ev\|_{\mathcal{C}^{s+1}},\qquad \|E_2\|_{\mathcal{C}^s}\leq K\beta^{-(2s+3)}.$$
Moreover, $E_2$ satisfies
\begin{equation}\label{eq:drift}
\begin{split} 
b(r)=&\int_0^1 E_2(\tilde\theta,\tilde
r)d\tilde\theta\\
=&\int_0^1 \Big(\E w(\tilde\theta,\tilde
r)-\pa_\theta\E v(\tilde\theta,\tilde
r)\E u(\tilde\theta,\tilde
r) \\
&+\partial_{\theta}S_1(\tilde\theta,\tilde r)\left(\partial_{\tilde r}
\Ev(\tilde\theta,\tilde
r)-\partial_{\theta}\Ev(\tilde\theta,\tilde
r)+\partial_{\theta}\Eu(\tilde\theta,\tilde
r)\right)\Big)d\tilde \theta.
\end{split}
\end{equation}
In particular, $b(r)$ satisfies $\|b\|_{\mathcal{C}^0}\leq K$.  

\item If $|\tilde r-p/q|\leq 2\beta$ for a given 
$p/q\in\mathcal{R}$, then
{\small \begin{equation}\label{NFnear}
\Phi^{-1}\circ\Ef\circ\Phi(\tilde\theta,\tilde r)=
\end{equation}
\begin{equation}
\begin{pmatrix}\tilde\theta+\tilde
r+\eps\left[\Eu\left(\tilde\theta,\frac{p}{q}\right)-\Ev\left(\tilde\theta,
\frac{p}{q}\right)+\Ev_ { p , q } \left(\tilde\theta ,\frac{ p}{q}\right)+
E_3(\tilde\theta)\right]+\mathcal{O}_s\left(\eps^{1+a},\eps\beta,\eps^3\beta^{
-(2s+4)}\right)\\
\tilde r+\eps\Ev_{p,q}(\tilde\theta,\tilde r)+
\eps^2 E_4(\tilde\theta,\tilde
r)+\mathcal{O}_s(\eps^{2+a},\eps^3\beta^{-(3s+5)})
\end{pmatrix},\nonumber 
\end{equation}}
where $\Ev_{p,q}$ is the $\mathcal{C}^l$ function defined as
\begin{equation}\label{defEvpq}
 \Ev_{p,q}(\tilde\theta,\tilde r)=\sum_{k\in \mathcal{R}_{p,q}}\Ev^k(\tilde
r)e^{2\pi ik\tilde\theta},
\end{equation}
and $E_3$ is the $\mathcal{C}^{l-1}$ function
\begin{equation}\label{defE3}
E_3(\tilde\theta)=-\sum_{k\not\in\mathcal{R}_\beta^{p,q}}\frac{i(\Ev^k)'(p/q)}{
2\pi
k}e^{2\pi i k\tilde\theta},
\end{equation}
where 
\begin{equation}\label{def:rpq}
\mathcal{R}_\beta^{p,q}=\{k\in\mathbb{Z}\,:\,k\neq0,\,|k|\leq2d,\,
kp/q\in\mathbb{Z}\}.
\end{equation}

Moreover, $E_4$ is a $\mathcal{C}^{l-1}$ function and there exists a constant 
$K$ such that for all $0\leq s\leq l-1$ one has
$$\|E_4\|_{\mathcal{C}^s}\leq K\beta^{-(2s+3)}.$$
\end{itemize}
Also, $\Phi$ is $\mathcal{C}^2$-close to the identity. 
More precisely, there exists a constant $M$ independent 
of $\eps$ such that
\begin{equation}\label{normPhi-Id}
\|\Phi-\textup{Id}\|_{\mathcal{C}^2}\leq M\eps.
\end{equation}
\end{thm}
\begin{cor}\label{corZeroDrift}
If the map \eqref{def:ExpectedMap} is area preserving and exact, 
\[b(r)\equiv
0.\]
\end{cor}
\begin{proof}[Proof of Corollary \ref{corZeroDrift}]
It is enough to recall the following two facts. First, expanding $\E
f^*(dr\wedge d\theta)-dr\wedge d\theta$ in $\eps$ and taking the first order,
one obtains that being $\E f$ area preserving   implies
$\partial_{\tilde r}
\Ev(\tilde\theta,\tilde
r)-\partial_{\theta}\Ev(\tilde\theta,\tilde
r)+\partial_{\theta}\Eu(\tilde\theta,\tilde
r)=0$. Second, expanding $\E
f^*(rd\theta)-rd\theta$ in $\eps$ and taking the first and second order, being
exact implies $\int_0^1 \E v(\tilde\theta,\tilde
r)d\tilde
r=0$ and 
\[
\int_0^1 \left(\E w(\tilde\theta,\tilde
r)-\pa_\theta\E v(\tilde\theta,\tilde
r)\E u(\tilde\theta,\tilde
r)\right)d\tilde
r=0.\
\]
\end{proof}

\begin{rmk} \label{transition-exponent}
 Notice that in the case $\beta=\eps^{1/11}$ and $s=0$
the remainder term 
$\mathcal{O}_0(\eps^2\beta^{-5})$ is dominated 
by $\mathcal{O}_0(\eps^{2+a})$ if $1/2<a<6/11$. 
\end{rmk}

\begin{proof}[Proof of Theorem \ref{thm:normal-form}]

Consider the canonical change defined implicitly by 
a given generating function 
$S(\theta,\tilde r)=\theta\tilde r+\eps S_1(\theta,\tilde r)$, that is
\[
\begin{split}
	\tilde \theta=&\partial_{\tilde r} S(\theta, \tilde r)=\theta+\eps
\partial_{\tilde r}S_1(\theta,\tilde r)\\
	r=&\partial_\theta S(\theta, \tilde r)=\tilde r+\eps \partial_\theta
S_1(\theta,\tilde r).
\end{split}
\]
We shall start by writing explicitly the first orders of 
the $\eps$-series of $\Phi^{-1}\circ\Ef\circ\Phi$.
If $(\theta,r)=\Phi(\tilde\theta,\tilde r)$ is the change given by the 
generating function $S$, then one has
\begin{eqnarray}\label{Phi}
\beal 
 \Phi(\tilde\theta,\tilde r)=\qquad \qquad \qquad \qquad
\qquad \qquad \qquad \qquad \qquad \qquad \qquad \qquad\\
\begin{pmatrix}\tilde\theta-\eps\partial_{\tilde r}S_1(\tilde\theta,\tilde
r)+\eps^2\partial_\theta\partial_{\tilde r}S_1(\tilde\theta,\tilde
r)\partial_{\tilde r}S_1(\tilde\theta,\tilde
r)+\mathcal{O}_s(\eps^3\|\partial_\theta^2\partial_{\tilde
r}S_1(\partial_{\tilde r}S_1)^2\|_{\mathcal{C}^s})\\
 \tilde r+\eps\partial_\theta S_1(\tilde\theta,\tilde r)-\eps^2\partial_\theta^2S_1(\tilde\theta,\tilde r)\partial_{\tilde r}S_1(\tilde\theta,\tilde r)+\mathcal{O}_s(\eps^3\|\partial_\theta^3S_1(\partial_{\tilde r}S_1)^2\|_{\mathcal{C}^s})
\end{pmatrix},
\enal 
\end{eqnarray}
and its inverse is given by
\begin{eqnarray}\label{Phiinverse}
\beal
 \Phi^{-1}(\theta,r)=\qquad \qquad \qquad \qquad
\qquad \qquad \qquad \qquad \qquad \qquad \qquad \qquad
\\
\begin{pmatrix}\theta+\eps\partial_{\tilde
r}S_1(\theta,r)-\eps^2\partial^2_{\tilde r}S_1(\theta,r)\partial_\theta
S_1(\theta,r)+\mathcal{O}_s(\eps^3\|\partial^3_{\tilde r}S_1(\partial_\theta
S_1)^2\|_{\mathcal{C}^s})
\\
 r-\eps\partial_\theta S_1(\theta,r)+\eps^2\partial_\theta\partial_{\tilde r}S_1(\theta,r)\partial_\theta S_1(\theta,r)+\mathcal{O}_s(\eps^3\|\partial_\theta\partial_{\tilde r}^2S_1(\partial_\theta S_1)^2\|_{\mathcal{C}^s})
\end{pmatrix}.
\enal 
\end{eqnarray}
One can see that
\begin{equation}\label{EfPhi}
\Ef\circ\Phi(\tilde\theta,\tilde r)=\begin{pmatrix}
\tilde\theta + \tilde r+ \eps
A_1+\eps^2A_2+\eps^3A_3+\mathcal{O}_s\left(\eps^{1+a}\right)\\
\tilde r+\eps
B_1+\eps^2B_2+\eps^3B_3+\mathcal{O}_s\left(\eps^{2+a}\right)
\end{pmatrix},
\end{equation}
where
\begin{equation}\label{defA1}
\begin{split}
 A_1=&\Eu(\tilde\theta,\tilde r)-\partial_{\tilde r}S_1(\tilde\theta,\tilde
r)+\partial_\theta S_1(\tilde\theta,\tilde r)\\
 A_2=&-\partial_\theta\Eu(\tilde\theta,\tilde r)\partial_{\tilde
r}S_1(\tilde\theta,\tilde r)+\partial_r\Eu(\tilde\theta,\tilde r)\partial_\theta
S_1(\tilde\theta,\tilde r)\\
 &+\partial_\theta\partial_{\tilde r}S_1(\tilde\theta,\tilde r)\partial_{\tilde
r}S_1(\tilde\theta,\tilde r) -\partial_\theta^2S_1(\tilde\theta,\tilde
r)\partial_{\tilde r}S_1(\tilde\theta,\tilde r),\\
 A_3=&\mathcal{O}_s(\|\partial_\theta^2\partial_{\tilde r}S_1(\partial_{\tilde
r} S_1)^2\|_{\mathcal{C}^s})+\mathcal{O}_s(\|\partial_\theta^3
S_1(\partial_{\tilde r} S_1)^2\|_{\mathcal{C}^s})\\
&+\mathcal{O}_s(\|\Eu\|_{\mathcal{C}^{s+1}}\|\partial_\theta
S_1\|_{\mathcal{C}^{s+1}}\|\partial_{\tilde r}S_1\|_{\mathcal{C}^s})\\
&+\mathcal{O}_s(\|\Eu\|_{\mathcal{C}^{s+2}}(\|\partial_\theta
S_1\|_{\mathcal{C}^s}+\|\partial_{\tilde
r}S_1\|_{\mathcal{C}^s})^2),
\end{split}
\end{equation}
and
\begin{equation}\label{defB2}
\begin{split}
 B_1=&\Ev(\tilde\theta,\tilde r)+\partial_\theta S_1(\tilde\theta,\tilde r),\\
 B_2=&\E w(\tilde \theta, \tilde r)-\partial_\theta\Ev(\tilde\theta,\tilde
r)\partial_{\tilde r}S_1(\tilde\theta,\tilde
r)\\ 
&+\partial_r\Ev(\tilde\theta,\tilde r)\partial_\theta
S_1(\tilde\theta,\tilde r)-\partial^2_\theta
S_1(\tilde\theta,\tilde r)\partial_{\tilde r}S_1(\tilde\theta,\tilde
r),\\
 B_3=&\mathcal{O}_s(\|\partial_\theta^3S_1(\partial_{\tilde
r}S_1)^2\|_{\mathcal{C}^s})+\mathcal{O}_s(\|\Ev\|_{\mathcal{C}^{s+1}}
\|\partial_\theta S_1\|_{\mathcal{C}^{s+1}}\|\partial_{\tilde
r}S_1\|_{\mathcal{C}^s})\\
 &+\mathcal{O}_s(\|\Ev\|_{\mathcal{C}^{s+2}}(\|\partial_\theta
S_1\|_{\mathcal{C}^s}+\|\partial_{\tilde
r}S_1\|_{\mathcal{C}^s})^2).
\end{split}
\end{equation}
Then, using \eqref{Phiinverse},
\begin{equation}\label{PhiinverseEfPhi}
\Phi^{-1}\circ\Ef\circ\Phi(\tilde\theta,\tilde r)=\begin{pmatrix}\tilde
\theta+\tilde r+\eps\hat A_1+\eps^2\hat A_2+\mathcal
O_s\left(\eps^{1+a}\right)\\\tilde r+\eps\hat B_1+\eps^2\hat B_2+\eps^3\hat
B_3+\mathcal
O_s\left(\eps^{2+a}\right)\end{pmatrix},
\end{equation}
where
\begin{equation}\label{defA1hat}
\begin{split}
 \hat A_1=&A_1+\partial_{\tilde r}S_1(\tilde\theta+\tilde r,\tilde
r),\\
 \hat A_2=&A_2+\eps A_3+\mathcal{O}_s(\|\partial_\theta\partial_{\tilde
r}S_1A_1\|_{\mathcal{C}^s})+\mathcal{O}_s(\|\partial_{\tilde
r}^2S_1B_1\|_{\mathcal{C}^s})\\
 &+\mathcal{O}_s(\|\partial_{\tilde r}^2S_1\partial_\theta
S_1\|_{\mathcal{C}^s}),
\end{split}
\end{equation}
and
\begin{equation}\label{defB2hat}
 \begin{split}
 \hat B_1=&B_1-\partial_\theta S_1(\tilde \theta+\tilde r,\tilde
r)\\
 \hat B_2=&B_2-\partial_\theta^2S_1(\tilde\theta+\tilde r,\tilde
r)A_1-\partial_{\tilde r}\partial_\theta S_1(\tilde\theta+\tilde r,\tilde
r)B_1\\
&+\partial_\theta\partial_{\tilde r}S_1(\tilde \theta+\tilde r,\tilde
r)\partial_\theta S_1(\tilde\theta+\tilde r,\tilde r),
\\
 \hat B_3=&B_3+\mathcal{O}_s(\|\partial_\theta\partial_{\tilde
r}^2S_1(\partial_\theta S_1)^2\|_{\mathcal{C}^s})\\
&+\mathcal{O}_s(\|\partial_\theta^2 S_1 (A_2+\eps
A_3)\|_{\mathcal{C}^s}+\|\partial_\theta\partial_{\tilde r} S_1
B_2\|_{\mathcal{C}^s})\\
&+\mathcal{O}_s(\|\partial_\theta^3S_1
A_1^2\|_{\mathcal{C}^s}+\|\partial_\theta^2\partial_{\tilde r}
S_1A_1B_1\|_{\mathcal{C}^s}+\|\partial_\theta\partial_{\tilde r}^2S_1
B_1^2\|_{\mathcal{C}^s})\\
&+\mathcal{O}_s(\|\partial_\theta^2\partial_{\tilde r}S_1A_1\partial_\theta
S_1\|_{\mathcal{C}^s}+\|\partial_\theta\partial_{\tilde r}^2
S_1B_1\partial_\theta S_1\|_{\mathcal{C}^s})\\
&+\mathcal{O}_s(\|\partial_\theta\partial_{\tilde r}S_1\partial_\theta^2S_1
A_1\|_{\mathcal{C}^s}+\|(\partial_\theta\partial_{\tilde
r}S_1)^2B_1\|_{\mathcal{C}^s}).
\end{split}
 \end{equation}
 
%
%
%

Now that we know the terms of order $\eps$ and $\eps^2$ of 
$\Phi^{-1}\circ\Ef\circ\Phi$, we  proceed to find a suitable $S_1(\theta,\tilde 
r)$ to make $\hat B_1$ as simple as possible. Ideally we would like 
that $\hat B_1=0$ by  solving the following equation whenever it is possible
\begin{equation}\label{eq:cohomological}
\partial_{\theta}S_1(\tilde\theta,\tilde r)+\Ev(\tilde\theta,\tilde
r)-\partial_{\theta}S_1(\tilde\theta+\tilde r,\tilde r)=0.
\end{equation}
One can  find a formal solution of this equation by solving the
corresponding equation for the Fourier coefficients. Write $S_1$
and $\Ev$ 
in their Fourier series
\begin{equation}
\label{eq:Genfunction}
S_1(\theta,\tilde r)=
\sum_{k\in\mathbb{Z}}S_1^k(\tilde r)e^{2\pi ik\theta},
\end{equation}
$$\Ev(\theta,r)=\sum_{\substack{k\in\mathbb{Z}\\0<|k|\leq d}}\Ev^k(r)e^{2\pi ik\theta}.$$
It is obvious that for $|k|>d$ and $k=0$ we can take $S_1^k(\tilde r)=0$. For 
$0<|k|\leq d$ we obtain the following homological equation for $S_1^k(\tilde r)$
\begin{equation} \label{eq:HomEq}
2\pi ikS_1^k(\tilde r)\left(1-e^{2\pi ik\tilde r}\right)+\Ev^k(r)=0.
\end{equation}
This equation cannot be solved if $e^{2\pi ik\tilde r}=1$, i.e. if $k\tilde 
r\in\mathbb{Z}$. We note that there exists a constant $L$, independent of 
$\eps$, $L<d^{-1}$, such that if $\tilde r\neq p/q$ satisfies
$$0<|\tilde r-p/q|\leq L$$
then $k\tilde r\not \in\mathbb{Z}$ for all $0<k\leq d$. Restricting 
ourselves to the domain $|\tilde r-p/q|\leq L$, we have that if 
$kp/q\not\in\mathbb{Z}$ equation \eqref{eq:HomEq} always has a solution, and if 
$kp/q\in\mathbb{Z}$ this equation has a solution except at $\tilde r=p/q$. 
Moreover, in the case that the solution exists, it is equal to:
$$S_1^k(\tilde r)=\frac{i\Ev^k(r)}{2\pi k\left(1-e^{2\pi ik\tilde r}\right)}.$$
We  modify this solution slightly to make it well defined also at $\tilde 
r=p/q$. To this end, let us consider a $\mathcal{C}^\infty$ function 
$\mu(x)$ such that
\begin{equation*}
\mu(x)=\left\{\begin{array}{rcl}
1 &\textrm{ if }& |x|\leq1,\\
0 &\textrm{ if }& |x|\geq2,
\end{array}\right.
\end{equation*}
and $0<\mu(x)<1$ if $|x|\in(1,2)$. Then we define
$$\mu_k(\tilde r)=\mu\left(\frac{1-e^{2\pi ik\tilde r}}{2\pi k
\beta
}\right),$$
and take
\begin{equation}\label{defS1k}
 S_1^k(\tilde r)=\frac{i\Ev^{k}(r)(1-\mu_k(\tilde r))}{2\pi k(1-e^{2\pi ik\tilde r})}.\end{equation}
This function is well defined since the numerator is identically zero in a 
neighbourhood of $\tilde r=p/q$, the unique zero 
of the denominator (if it is a zero indeed, that is, if 
$k\in\mathcal{N}\cap q\mathbb{Z}$, see \eqref{def:FourierSupp}). More precisely,
we claim that
\begin{equation}\label{valuesmuk}
 \mu_k(\tilde r)=\left\{\begin{array}{ccl}
 1&\textrm{ if }& k\in\mathcal{N}\cap q\mathbb{Z} \ \textrm{ and }\ |\tilde
r-p/q|\leq 
\beta
/2,\\
0&\textrm{ if }& k\in\mathcal{N}\cap q\mathbb{Z} \ \textrm{ and }\ |\tilde
r-p/q|\geq 3\beta
,
\\
0&\textrm{ if }&k\not\in\mathcal{N}\cap q\mathbb{Z}.
\end{array}\right.
\end{equation}
Indeed if $k\in\mathcal{N}\cap q\mathbb{Z}$ there exists a constant $M$ 
independent of $\tilde r$ and $\eps$ such that
$$
\frac{1}{\beta
}|\tilde r-p/q|(1-M|\tilde r-p/q|)
\leq
\left|\frac{1-e^{2\pi ik\tilde r}}{2\pi k \beta
}\right|\leq\frac{1}{ \beta
}|\tilde r-p/q|(1+M|\tilde r-p/q|).
$$
Then, on the one hand, if $k\in\mathcal{N}\cap q\mathbb{Z}$ and 
$|\tilde r-p/q|\leq \beta
/2$ we have:
$$
\left|\frac{1-e^{2\pi ik\tilde r}}{2\pi k \beta
}\right|\leq\frac{1}{2}+\frac{M}{4} \beta
<1,$$
for $\beta$ sufficiently small, and thus $\mu_k(\tilde r)=1$. 
On the other hand, if $|\tilde r-p/q|\geq 3\beta
$ then
$$
\left|\frac{1-e^{2\pi ik\tilde r}}{2\pi k\beta
}\right|\geq 3-9M \beta
>2,$$
for $\beta$ sufficiently small, and thus $\mu_k(\tilde r)=0$. 
Finally, if $k\not\in \mathcal{N}\cap q\mathbb{Z}$ then
$$\left|\frac{1-e^{2\pi ik\tilde r}}{2\pi k \beta
}\right|\geq \frac{M}{\beta}
>2$$
for $\beta$ sufficiently small and then we also have 
$\mu_k(\tilde r)=0$.

Now we proceed to check that the first order terms of \eqref{PhiinverseEfPhi} take the form \eqref{NFfar} if 
$|\tilde r-p/q|\geq 3\beta$
and \eqref{NFnear} if 
$|\tilde r-p/q|\leq \beta/2$.
On the one hand, by definitions in  \eqref{defS1k} of 
the coefficients $S_1^k(\tilde r)$ and in \eqref{defB2hat} 
of $\hat B_1$, we have
$$
\hat B_1=\sum_{0<|k|\leq d}
\mu_k(\tilde r)\Ev^k(\tilde r)e^{2\pi i k\tilde\theta}.$$
Then, recalling \eqref{valuesmuk} we obtain
\begin{equation}\label{valuesB1hat}
 \hat B_1=\left\{\begin{array}{lcl}0&\quad
 \textrm{ if }&|\tilde r-p/q|\geq 3\beta
 \\ 
 \displaystyle\sum_{k\in\mathcal{N}\cap q\mathbb{Z}}
 \Ev^k(\tilde r)e^{2\pi i k\tilde\theta}= 
 \Ev_{p,q}(\tilde\theta,\tilde r)&\quad
 \textrm{ if }&|\tilde r-p/q|\leq \beta/2
 .\end{array}\right. 
\end{equation}
where we have used the definition \eqref{defEvpq} of 
$\Ev_{p,q}(\tilde\theta,\tilde r)$. On the other hand, from 
the definition \eqref{defS1k} of $S_1^k(\tilde r)$ one can check that
\begin{eqnarray*}
 &&-\partial_{\tilde r}S_1(\tilde\theta,\tilde r)+
 \partial_{\tilde r}S_1(\tilde\theta+\tilde r,\tilde r)\\
 &&=-\partial_\theta S_1(\tilde\theta+\tilde r,\tilde
r)-\sum_{0<|k|\leq d}\frac{i(\Ev^k)'(\tilde r)(1-\mu_k(\tilde r))+
 i\Ev^k(\tilde r)\mu_k'(\tilde r)}{2\pi k}e^{2\pi i k\tilde\theta}.
\end{eqnarray*}
Recalling definitions \eqref{defA1hat} of $\hat A_1$ and 
\eqref{defB2hat} of $\hat B_1$, this implies that
\begin{equation}\label{hatA1rewrite}
\begin{split}
\hat A_1=&\Eu(\tilde\theta,\tilde r)-\Ev(\tilde\theta,\tilde r)+\hat B_1\\
&-\sum_{0<|k|\leq d}\frac{i(\Ev^k)'(\tilde r)(1-\mu_k(\tilde r))+i\Ev^k(\tilde
r)\mu_k'(\tilde r)}{2\pi k}e^{2\pi i k\tilde\theta}.
\end{split}
\end{equation}
Then we use \eqref{valuesB1hat} and \eqref{valuesmuk} 
again, noting that $\mu'_k(\tilde r)=0$ in both regions 
$|\tilde r-p/q|\geq 3\beta$
and $|\tilde r-p/q|\leq \beta/2$,
Moreover, we note that for $|\tilde r-p/q|\leq \beta/2$.
$$
\Ev_{p,q}(\tilde\theta,\tilde r)=
\Ev_{p,q}(\tilde\theta,p/q) +\mathcal{O}(\beta
),$$
$$(\Ev^k)'(\tilde r)=(\Ev^k)'(p/q)+\mathcal{O}(
\beta
).$$
Define 
\begin{equation}\label{defE1}
E_1(\tilde\theta,\tilde r)=-\sum_{0<|k|\leq d}
\frac{i(\Ev^k)'(\tilde r)}{2\pi k}e^{2\pi i k\tilde\theta}.
\end{equation}
Then  the same holds for $\Eu(\tilde\theta,\tilde r)$ and $\Ev(\tilde\theta,\tilde r)$: recalling definition 
\eqref{defE3} of $E_3$, equation \eqref{hatA1rewrite} yields
\begin{equation}\label{valuesA1hat}
 \hat A_1=\left\{\begin{array}{lcl}\Eu(\tilde\theta,\tilde r)-\Ev(\tilde\theta,\tilde r)+E_1(\tilde\theta,\tilde r)&\,
\textrm{ if }&|\tilde r-p/q|\geq 3\beta
 ,\\ \Delta \E(\tilde\theta,p/q)
+\Ev_{p,q}(\tilde\theta)+E_3(\tilde\theta)+\mathcal{O}(\beta) &\,\textrm{ if
}&|\tilde r-p/q|\leq \beta/2,\end{array}\right.
\end{equation}
where $\Eu(\tilde\theta,p/q)-\Ev(\tilde\theta,p/q)=\Delta \E(\tilde\theta,p/q)$. 
In conclusion, by \eqref{valuesA1hat} and \eqref{valuesB1hat} 
we obtain that the first order terms of \eqref{Phiinverse} coincide 
with the first order terms of \eqref{NFfar} and \eqref{NFnear} in 
each region.

For the $\eps^2-$terms we rename $\hat B_2$ in the following way
\begin{equation}\label{defE2}
\begin{split}
E_2(\tilde\theta,\tilde r)&=\hat B_2|_{\{|\tilde r-p/q|\geq
3\beta
\}}, \\
E_4(\tilde\theta,\tilde r)&=\hat B_2|_{\{|\tilde r-p/q|\leq\beta/2
\}}.
\end{split}
\end{equation}
Now we  see that $E_2$ satisfies \eqref{eq:drift}. To avoid 
long notation, in the following we do not write explicitly that 
expressions $A_i$, $B_i$, $\hat A_i$ and $\hat B_i$ are restricted 
to the region $\{|\tilde r-p/q|\geq 3 \beta
\}$.  We note that since in 
this region we have $\hat B_1=0$ by \eqref{valuesB1hat}, recalling 
the definition \eqref{defB2hat} of $\hat B_1$ it is clear that 
$B_1=\partial_\theta S_1(\tilde\theta+\tilde r,\tilde r)$. Hence, from 
definition \eqref{defB2hat} of $\hat B_2$ it is straightforward to see that
\begin{equation}\label{hatB2simple}
\hat B_2=B_2-\partial_\theta^2S_1(\tilde\theta+\tilde r,\tilde r)A_1.
\end{equation}
Recalling that $\hat A_1=A_1+\partial_{\tilde r}S_1(\tilde\theta+\tilde
r,\tilde r)$ and  using the definition of $A_1$ in
\eqref{defA1} and  the definition \eqref{defB2} of $B_2$,
\begin{equation}\label{B2hatequality}
\begin{split} 
E_2(\tilde\theta,\tilde r)=&\,\hat B_2|_{\{|\tilde r-p/q|\geq3\beta
\}}\\
=&\,\E w(\tilde\theta,\tilde r)-\partial_\theta\Ev(\tilde\theta,\tilde
r)\partial_{\tilde r}S_1(\tilde\theta,\tilde r) \\
&+\partial_r\Ev(\tilde\theta,\tilde r)\partial_\theta S_1(\tilde\theta,\tilde
r)-\partial^2_\theta S_1(\tilde\theta,\tilde r)\partial_{\tilde
r}S_1(\tilde\theta,\tilde r) \\
&-\partial_\theta^2S_1(\tilde\theta+\tilde r,\tilde
r)\left[\Eu(\tilde\theta,\tilde r)+\partial_\theta S_1(\tilde\theta,\tilde
r)-\partial_r S_1(\tilde\theta,\tilde r)\right]\\
&-\partial_{\theta}\partial_{\tilde r}S_1(\tilde\theta+\tilde r,\tilde
r)\left[\E v(\tilde\theta,\tilde r)+\partial_\theta S_1(\tilde\theta,\tilde
r)-\partial_\theta S_1(\tilde\theta+\tilde r,\tilde r)\right]. 
\end{split}
\end{equation}
Since, for $|\tilde r-p/q|\geq3\beta$,  $S_1$ satisfies
\eqref{eq:cohomological}, the last row of the definition of $E_2$ vanishes and
the same happens with 
\[
 \begin{split}
  -\partial_\theta \E v(\tilde\theta,\tilde r)\pa_{\tilde r}
S_1(\tilde\theta,\tilde r)-\pa_\theta^2S_1(\tilde\theta,\tilde r)\pa_{\tilde
r}S_1(\tilde\theta,\tilde r)+\pa_\theta^2 S_1(\tilde\theta+\tilde r,\tilde
r)\pa_{\tilde r}S_1(\tilde\theta,\tilde r)&=\\
\pa_{\tilde r}S_1(\tilde\theta,\tilde r)\left( -\partial_\theta  \E
v(\tilde\theta,\tilde r)-\pa_\theta^2S_1(\tilde\theta,\tilde r)+\pa_\theta^2
S_1(\tilde\theta+\tilde r,\tilde
r)\right)&=0.
 \end{split}
\]
Therefore,
\[
 \begin{split} 
b(\tilde r)=&\int_0^1 E_2(\tilde\theta,\tilde r)d\tilde\theta\\
=&\int_0^1 \Big(\E w(\tilde\theta,\tilde
r)+\partial_{\tilde r}\Ev(\tilde\theta,\tilde
r)\partial_{\theta}S_1(\tilde\theta,\tilde r) \\
&-\partial^2_\theta S_1(\tilde\theta+\tilde r,\tilde
r)\left(\Eu(\tilde\theta,\tilde r)+\partial_\theta S_1(\tilde\theta,\tilde
r)\right)\Big)d\tilde \theta.
\end{split}
\]
Using $\partial^2_\theta S_1(\tilde\theta+\tilde r,\tilde
r)=\partial^2_\theta S_1(\tilde\theta,\tilde
r)+\partial_\theta\Ev(\tilde\theta,\tilde
r)$ and taking into account that $\int_0^1  \partial^2_\theta
S_1(\tilde\theta,\tilde
r)\partial_\theta S_1(\tilde\theta,\tilde
r)d\tilde\theta=0$, we have that 
\[
 \begin{split}
  b(\tilde r)=&\int_0^1 \Big(\E w(\tilde\theta,\tilde
r)+\partial_{\tilde r}\Ev(\tilde\theta,\tilde
r)\partial_{\theta}S_1(\tilde\theta,\tilde r)-\pa_\theta \Ev(\tilde\theta,\tilde
r)\E u(\tilde\theta,\tilde
r) \\
&-\partial_\theta \Ev(\tilde\theta,\tilde
r)\partial_{\theta}S_1(\tilde\theta,\tilde
r)-\partial^2_{\theta}S_1(\tilde\theta,\tilde r)\Eu(\tilde\theta,\tilde
r)\Big)d\tilde \theta.
 \end{split}
\]
Integrating by parts, we obtain \eqref{eq:drift}.
%
%

We note that, from the definition \eqref{defS1k} of the Fourier coefficients  of 
$S_1$, it is clear that $S_1$ is $\mathcal{C}^l$ with respect to $r$. Since it 
just has a finite number of nonzero coefficients, it is analytic with respect to 
$\theta$. Then, from the definitions \eqref{defE2} of $E_2$ and 
$E_4$ and the expression $\eqref{defB2hat}$ of $\hat B_2$, it is clear that both 
$E_2$ and $E_4$ are $\mathcal{C}^{l-1}$.

Finally we  bound the $\mathcal{C}^0$-norms of the functions $E_2$, $b(r)$ and 
$E_4$ and also the error terms. To that aim, we bound the 
$\mathcal{C}^l$ norms of $S_1$ and its derivatives. We will use Lemma 
\ref{lemClnorms} and proceed similarly as in \cite{BKZ}. We note that
\begin{enumerate}
	\item If $\mu_k(\tilde r)\neq 1$ we have $|1-e^{2\pi ik\tilde r}|>M 
	\beta
	|k|$, and thus
	$$\left|\frac{1}{1-e^{2\pi ik\tilde r}}\right|<M^{-1}
		\beta^{-1}
		|k|^{-1}.$$
	\item Then, using that $\|f\circ g\|_{\mathcal{C}^l}\leq
C\|f{|_{\textrm{Im}(g)}}\|_{\mathcal{C}^l}\left(1+\|g\|_{\mathcal{C}^l}
^l\right)$, we get that
	$$\left\|\frac{1}{1-e^{2\pi ik\tilde r}}\right\|_{\mathcal{C}^l}\leq 
M	\beta
	^{-(l+1)}|k|^{-(l+1)},$$
	for some constant $M$, not the same as item 1.
	\item Using the rule for the norm of the composition again 
and the fact that $\|\mu\|_{\mathcal{C}^l}$ is bounded 
independently of $\beta$, we get
	$$\|\mu_k(\tilde r)\|_{\mathcal{C}^l}\leq M\beta^{-l}
	|k|^{-l},$$
	for some constant $M$, and the same bound is obtained 
for $\|1-\mu_k(\tilde r)\|_{\mathcal{C}^l}$.
\end{enumerate}
Using items $2$ and $3$ above and the fact that $\|\Ev^k\|_{\mathcal{C}^l}$ are
bounded, we get that
{\small \begin{eqnarray*}
\left\|\partial_{\tilde r^\alpha}\left[\frac{1-\mu_k(\tilde r)i\Ev^k(\tilde r)}{2\pi k(1-e^{2\pi ik\tilde r})}\right]\right\|_{\mathcal{C}^0}&\leq&M_1\sum_{\alpha_1+\alpha_2=\alpha}\frac{1}{2\pi|k|}\|1-\mu_k(\tilde r)\|_{\mathcal{C}^{\alpha_1}}\left\|\frac{1}{1-e^{2\pi ik\tilde r}}\right\|_{\mathcal{C}^{\alpha_2}}\\
&\leq& M_2\beta^{-(\alpha+1)}
|k|^{-\alpha-2}.
\end{eqnarray*}}
Then, by item $2$ of Lemma \ref{lemClnorms}, we obtain
$$ \|S_1\|_{\mathcal{C}^l}\leq M\beta^{-(l+1)}
.$$
One can also see that $\|\partial_{\tilde r} S_1\|_{\mathcal{C}^l}\leq
M\|S_1\|_{\mathcal{C}^{l+1}}$ and $\|\partial_\theta S_1\|_{\mathcal{C}^l}\leq
M\|S_1\|_{\mathcal{C}^l}$. In general, one has
\begin{equation}\label{boundderivsS1}
 \|\partial_\theta^n\partial_{\tilde r}^m S_1\|_{\mathcal{C}^l}\leq M\beta^{-(l+m+1)}
 .
\end{equation}
Now, recalling definitions \eqref{defE2} of $E_2$ and $E_4$,
and using  \eqref{B2hatequality},
bound \eqref{boundderivsS1} implies that for $0\leq s\leq l-1$ there exists some
$K>0$ independent of $\eps$ and $\beta$ such that
$$\|E_2\|_{\mathcal{C}^s}\leq K\beta
^{-(2s+3)},\qquad \|E_4\|_{\mathcal{C}^0}\leq K\beta
^{-(2s+3)}.$$
To bound the $\CCC^s$ norm, $0\leq s\leq l-1$,  of $b(r)$ in  \eqref{eq:drift}, 
 we use again \eqref{boundderivsS1} to obtain
$$
\|b\|_{\mathcal{C}^s}\leq K\beta^{-(s+1)}
.$$
Similarly, and taking into account that for $n=1,2$ we have
$$
\|\Eu\|_{\mathcal{C}^{s+n}}\leq K,\|\Ev\|_{\mathcal{C}^{s+n}}\leq K,
$$ 
because  $s\leq l-2$,  the error term in 
the equation for $\tilde r$ satisfies
\begin{equation}\label{errorr}
 \eps^3\hat B_3=\mathcal{O}_s(\eps^3\beta^{-(3s+5)}),
\end{equation}
and the error terms for the equation of $\tilde \theta$,
\begin{equation}\label{errortheta}
 \eps^2\hat A_2=\mathcal{O}_s(\eps^2\beta^{-(2s+4)}).
\end{equation}
This completes the proof for the normal forms \eqref{NFfar} and \eqref{NFnear} 
(in the latter case, we have to take into account 
the extra error term of order $\mathcal{O}(\eps^{1+a})$ caused 
by the $\beta$
--error term in \eqref{valuesA1hat}).

To prove \eqref{normPhi-Id}, we just need to recall \eqref{Phi} and 
use \eqref{boundderivsS1}. Then one obtains
$$\|\Phi-\textrm{Id}\|_{\mathcal{C}^2}\leq M'\eps\|S_1\|_{\mathcal{C}^3}.
$$
\end{proof}

From now on we  consider that our deterministic system is 
in  normal form, and we  drop tildes.

\section{Analysis of the Martingale problem in the strips
of each type}

After performing the change to normal form (Theorem \ref{thm:normal-form}), the 
$n$-th iteration of the original map (see \eqref{mapthetanrn}), becomes both in
the Totally Irrational and Imaginary Rational zones of the form
\begin{equation}\label{eq:NRmap-n}
\begin{split}
\theta_n=&\displaystyle\theta_0+nr_0+\mathcal{O}(n^2\eps),\\
r_n=&\displaystyle r_0+\eps\sum_{k=0}^{n-1}\omega_k[v(\theta_k,r_k)+\eps
v_2(\theta_k,r_k)]\\
&+\eps^2\sum_{k=0}^{n-1}E_2(\theta_k,r_k)+\mathcal{O}(n\eps^{
2+a}),
\end{split}
\end{equation}
where $v_2(\theta,r)$ is a given function which can be written explicitly in
terms of $v(\theta,r)$ and $S_1(\theta,r)$.

\subsection{The TI case}\label{sec:TI-case}

Recall that we have defined $\ga\in (4/5,4/5+1/40)$ and $\nu=1/4$. A 
strip $I_\ga$ is a  totally irrational segment if $p/q\in I_\ga$,
then $|q|>\eps^{-b},$ where $0<2b<\nu$ and that we define
$b=(\nu-\rho)/2$ for a certain $0<\rho<\nu$. In the following we shall 
assume that $\rho$ satisfies an extra condition, which ensures that certain
inequalities are satisfied. These inequalities involve the degree of
differentiability of certain $\mathcal{C}^l$ functions. Assume that $l\geq 6$. 
Then, there exists a constant $R>0$ such that
\begin{equation}\label{constantR}
R= \frac{l-5}{l-2}>0,\quad \textrm{ for all }\quad l\geq 6.
\end{equation}
We choose $\rho$, satisfying
\begin{equation}\label{conditionbeta}
\quad \rho=R\nu.
\end{equation}
%
%

\begin{lem}\label{lemmasigma2}
Fix $\tau\in (0,1/40)$ and let $g$ be a $\mathcal{C}^l$ function, $l\ge 6$.
Suppose $r^*$ satisfies the following condition: if for some 
rational $p/q$ we have $|r^*-p/q|<\eps^\nu$, then $|q|>\eps^{-b}$. 
Then,
for $\eps>0$ small enough there is 
$N\le \eps^{-(\nu+b+2\tau)}$
such that for some $K$ independent of $\eps$ and any $\theta^*$ we have
\[\left|N\,\int_0^1g(\theta,r^*)d\theta-\sum_{k=0}^{N-1}g(\theta^*+kr^*,
r^*)\right|\leq K\eps^{\tau}. \]
\end{lem}
\begin{proof} Denote $g_0(r)=\int_0^1g(\theta,r)d\theta$. 
Expand $g(\theta,r)$ in its Fourier series, i.e.
$$
g(\theta,r)=g_0(r)+\sum_{m\in \Z\setminus \{0\}} g_m(r) e^{2\pi im\theta}
$$
for some $g_m(r):\mathbb{R}\to \mathbb{C}$. Then we have
\begin{equation}\label{rewritesum}
\begin{split}
\sum_{k=0}^{N-1}&(g(\theta^*+kr^*,r^*)-g_0(r^*))=
\sum_{k=0}^{N-1}\sum_{m\in \Z\setminus \{0\}}g_m(r^*) e^{2\pi
im(\theta^*+kr^*)}\\
=&\sum_{k=0}^{N-1}\sum_{1\leq|m|\leq[\eps^{-b}]}g_m(r^*) e^{2\pi
im(\theta^*+kr^*)}+\sum_{k=0}^N\sum_{|m|\geq[\eps^{-b}]}g_m(r^*) e^{2\pi
im(\theta^*+kr^*)}\\
=&\sum_{1\leq|m|\leq[\eps^{-b}]}g_m(r^*)e^{2\pi im\theta^*}\sum_{k=0}^{N-1}
e^{2\pi imkr^*}+\sum_{k=0}^{N-1}
\sum_{|m|\geq[\eps^{-b}]}g_m(r^*) e^{2\pi im(\theta^*+kr^*)}\\
=&\sum_{1\leq|m|\leq[\eps^{-b}]}g_m(r^*)e^{2\pi im\theta^*}
\frac{e^{2\pi iNmr^*}-1}{e^{2\pi imr^*}-1}+
\sum_{k=0}^{N-1}\sum_{|m|\geq[\eps^{-b}]}
g_m(r^*) e^{2\pi im(\theta^*+kr^*)}.\\
\end{split}
\end{equation}
To bound the first sum in \eqref{rewritesum} we distinguish into the following
cases
\begin{itemize}
\item If $r^*$ is rational $p/q$, we know that $|q|>\eps^{-b}$. 
\begin{itemize}
\item 
If $|q|\le \eps^{-(\nu+b+2\tau)}$, then pick $N=|q|$ and the first sum
vanishes. 

\item If $|q|>\eps^{-(\nu+b+2\tau)}$, then by definition of $r^*$ for 
any $s/m$ with $|m|<\eps^{-b}$ we have
or $|m r^*-s|>\eps^\nu$. By the pigeon hole principle 
there exist integers $0<N=\tilde q<\eps^{-(\nu+b+2\tau)}$ and $\tilde p$ such
that
$|\tilde qr^*-\tilde p|\le 2 \eps^{\nu+b+2\tau}$. 
\end{itemize}
\item If $r^*$ is irrational, consider a continuous fraction expansion 
$p_n/q_n\to r^*$ as $n\to \infty$. Choose $p'/q'=p_n/q_n$ with
$n$ such that $q_{n+1}>\eps^{-(\nu+b+2\tau)}$. This implies that 
$|q'r^*-p'|<1/|q_{n+1}|\le \eps^{\nu+b+2\tau}$. 

The same argument as above 
shows that for any value $|m|<\eps^{-b}$ we have
\hbox{$|m r^*-s|>\eps^\nu$}.
\end{itemize}

Let $N$ be as above. Then, since $|m|\leq \eps^{-b}$, 
\[\left|\sum_{1\leq|m|\leq[\eps^{-b}]}\ g_m(r^*)\ e^{2\pi im\theta^*}\
\frac{e^{2\pi iNmr^*}-1}{e^{2\pi imr^*}-1}\right|\leq 2
\eps^{\tau}\sum_{1\leq|m|\leq[\eps^{-b}]}|g_m(r^*)|.\]
Since $g(\theta,r)$ is $\mathcal{C}^l$,  its Fourier
coefficients satisfy $|g_m(r^*)|\le C|m|^{-l},\ m\ne 0$. Thus we can bound the
first sum in \eqref{rewritesum} by
\[
\left|\sum_{1\leq|m|\leq[\eps^{-b}]}\ g_m(r^*)\ e^{2\pi im\theta^*}\
\frac{e^{2\pi iNmr^*}-1}{e^{2\pi imr^*}-1}\right|
\leq K\eps^{\tau}\sum_{1\leq|m|\leq[\eps^{-b}]}\frac{1}{m^2
}\leq K\eps^\tau.
\]
To bound the second sum, we use again the bound for the Fourier coefficients
$g_m(r^*)$
\begin{equation}\label{secondsum}
\left|\sum_{k=0}^N\sum_{|m|\geq[\eps^{-b}]}g_m(r^*) e^{2\pi
im(\theta+kr^*)}\right|\leq N\sum_{|m|\geq[\eps^{-b}]}\frac{1}{m^l}
 \le K \eps^{(l-1)b-(\nu+b+2\tau)}. 
\end{equation}
Taking into account that   $b=(\nu-\rho)/2$, $\rho\leq R\nu$
where $R=(l-5)/(l-2)$, $\nu=1/4$ and $\tau\in (0,1/40)$, one obtains
\[
 \left|\sum_{k=0}^N\sum_{|m|\geq[\eps^{-b}]}g_m(r^*) e^{2\pi
im(\theta+kr^*)}\right|\leq
K\eps^{\frac{\nu}{2}-2\tau}\leq K\eps^{\tau}
\]
\end{proof}

Fix a totally irrational strip $I_\ga$ and let $(\theta_0,r_0)\in I_\ga$. 
Recall that 
$n_\ga\leq n\leq 
s\eps^{-2}$
is either the exit time from $I_\ga$, that is the first number such that
$(\theta_{n_\ga+1},r_{n_\ga+1})\not\in I_{\ga}$  or 
$n_\ga=n$ the final time. 

\begin{lem}\label{lemma:TI:exittime}
Fix $\ga\in (4/5,4/5+1/40)$. Then, there
exists a constant  $C>0$ such that, 
\begin{itemize}
\item For any $\de\in (0, 2(1-\ga))$ and $\eps>0$ small
enough,
\[
\Prob\{n_\ga<\eps^{-2(1-\ga)+\delta}\}\leq
e^{-\frac{C}{\eps^{\delta}}}.
\]
\item For any   $\de>0$ and $\eps>0$ small
enough,
\[
\Prob\{\eps^{-2(1-\ga)-\delta}<n_\ga<s\eps^{-2}\}\leq
e^{-\frac{C}{\eps^{\delta}}}.
\]
\end{itemize}
\end{lem}
\begin{proof}
We first prove the second statement. Let $\wt n_\ga=[\eps^{-2(1-\ga)}],$ 
$n_\delta=[\eps^{-\delta}]$, and
$n_i=in_\ga$. Then,
\begin{equation}\label{probN-D3}
\begin{split}
\Prob\left\{n_\ga>\eps^{-2(1-\ga)-\delta}\right\}\leq&\,\Prob\left\{|r_{n_{i+1}}
-r_ { n_i } |\leq
\eps^\ga\,\textrm{ for all }i=0,\dots, n_\delta-1\right\}\\
 \leq&\prod_{i=0}^{n_\delta}\Prob\left\{|r_{n_{i+1}}-r_{n_i}|\leq
\eps^{\ga}\right\}.
 \end{split}
 \end{equation}
We have that
\[r_{n_{i+1}}=r_{n_i}+\eps
\sum_{k=0}^{\wt n_\ga-1}\omega_k
v(\theta_{n_i+k},r_{n_i+k})+\mathcal{O}(\wt n_\ga\eps^2).\]
Taking also into account that 
$\theta_{n_i+k}=\theta_{n_i}+kr_{n_i}+\mathcal{O}
(\wt n_\ga^2\eps)$ for $0\leq k\leq \wt n_\ga$, we can write
 \begin{equation}\label{eqHni}
r_{n_{i+1}}=r_{n_i}+\eps
\sum_{k=0}^{\wt n_\ga-1}\omega_k v\left(\theta_{n_i}+k 
r_{n_i},r_{n_i}\right)+\mathcal{O}(\wt n_\ga^3\eps^2). 
 \end{equation}
Define
\begin{equation}\label{def:xi}
\xi=\frac{1}{\sqrt{\wt n_\ga}}\sum_{k=0}^{\wt n_\ga-1}\omega_k
v(\theta_{n_i}+kr_{n_i},r_{n_i}).
\end{equation}
For $\wt n_\ga$ sufficiently large (i.e., for $\eps$ sufficiently small), one
has that $\xi$ converges in distribution to a normal random variable
$\mathcal{N}(0,\sigma^2(\theta_{n_i},r_{n_i}))$ with
\[
\sigma^2(\theta_{n_i},r_{n_i})=\frac{1}{\wt
n_\ga}\sum_{k=0}^{\wt n_\ga-1}v^2(\theta_{
n_i}+kr_{n_i},r_{n_i}).\]
Then it is enough to use Lemma \ref{lemmasigma2} (if $\wt n_\ga\geq
\eps^{-(\nu+b+2\tau)}$, it is enough to split the sum into several sums) and
use Hypothesis {\bf [H1]} to ensure that
$\sigma^2(\theta_{n_i},r_{n_i})\geq K>0$ for some constant $K$. Then
\eqref{eqHni} yields
\[
r_{n_{i+1}}-r_{n_i}=\eps \wt n_\ga^{1/2}\xi+\mathcal{O}
(\wt n_\ga^3\eps^2)\]
Then, using that $\ga\in (4/5,4/5+1/40)$,
\[
\Prob\{|r_{n_{i+1}}-r_{n_i}|\leq
\eps^{\ga}\}=\Prob\{|\xi+\mathcal{O}(\eps^{5\ga-4})|\leq
1\}\leq\Prob\{|\xi|\leq 2\}.\]
Since $\xi$ converges in
distribution to $\mathcal{N}(0,\sigma^2(\theta_{n_i},r_{n_i}))$ and
$\sigma^2(\theta_{n_i},r_{n_i})\geq K>0$, one has
 \[\Prob\{|r_{n_{i+1}}-r_{n_i}|\leq \eps^{\ga}\}\leq\rho,\]
 for some $0<\rho<1$. Using this in \eqref{probN-D3} one obtains the claim of
the lemma with $C=-\log\rho>0$.

For the first statement, note that $\Prob\{n_\ga<\eps^{-1-\ga}\}=0$ since
$|r_{k+1}-r_k|\leq 2\eps$ and therefore  one needs at least 
$\lceil\eps^{-1-\ga}/2\rceil$ iterations. Thus, we only need to analyze 
$\Prob\{\eps^{-1-\ga}/2\leq n_\ga<\eps^{-2(1-\ga)+\delta}\}$, which is 
equivalent to 
\[
 \Prob\{\exists\, n\in [\eps^{-1-\ga}/2,\eps^{-2(1-\ga)+\delta}): 
|r_n-r_0|\geq \eps^\ga\}.
\]
Proceeding as before, for $\eps>0$ small enough,
\[
 \begin{split}
 \Prob\left\{ |r_{n}-r_0|\geq \eps^\ga\right\}&\leq 
\Prob\left\{\left|\eps\sum_{k=0}^{n-1}\omega_kv(\theta_0+r_0k,
r_0)+\OO\left(\eps^2n^3\right)\right|\geq\eps^\ga\right\} \\
&\leq \Prob\left\{\left|\xi+\OO(\eps 
n^{5/2})\right|\geq\eps^{\ga-1}n^{-1/2}\right\}
 \end{split}
\]
where $\xi$ is the function defined in \eqref{def:xi} with $n_i=0$. Now, using 
that $\ga\in (4/5,4/5+1/40)$ and $n\in
[\eps^{-1-\ga}/2,\eps^{-2(1-\ga)+\delta})$ 
we have that 
\[
\Prob\left\{\left|\xi+\OO\left(\eps 
n^{5/2}\right)\right|\geq\eps^{\ga-1}n^{-1/2}\right\}\leq 
\Prob\left\{|\xi|\geq\frac{\eps^{-\de/2}}{2}\right\}
\]
By Lemma \ref{main lemma} and hypothesis \textbf{H1},  $\xi$ 
converges to a normal random  variable with $\sigma^2>0$ 
(with lower bound independent of $\eps$) as 
$\eps\rightarrow 0$. Thus,  
\[
 \Prob\left\{ |r_{n}-r_0|\geq \eps^\ga\right\} \leq e^{-\frac{C}{\eps^\de}}
\]
for some $C>0$ independent of $\eps$. Then, since $\sharp
[\eps^{-1-\ga}/2,\eps^{-2(1-\ga)+\delta})\sim \eps^{-2(1-\ga)+\delta}$, 
\[
 \Prob\{\exists\, n\in [\eps^{-1-\ga}/2,\eps^{-2(1-\ga)+\delta}): 
|r_n-r_0|\geq \eps^\ga\}\leq e^{-\frac{C}{\eps^\de}},
\]
taking a smaller $C>0$.
\end{proof}

Now we state the main lemma of this section which shows the convergence of the 
random map to a diffusion process in the strip $I_{\ga}$. To this end,
we define the  functions $b$  and $\sigma$ as in \eqref{eq:drift-variance}. 

\begin{lem}\label{lemmaexpectation}
Let $\nu$,  $b=(\nu-\rho)/2$ and $\rho$ 
satisfy  \eqref{conditionbeta} and $\ga\in (4/5,4/5+1/40)$. Take 
$f:\mathbb{R}\rightarrow\mathbb{R}$ be any $\mathcal{C}^l$ function with $l\ge 
3$ and $\|f\|_{\CCC^3}\leq C$ for some constant $C>0$ independent of $\eps$.
Then there exists $\zeta>0$ such that
\begin{eqnarray*}
&&\E\Bigg(f(r_{n_\ga})- \eps^2
\sum_{k=0}^{n_\ga-1}\left(b(r_k)f'(r_k)+\frac{\sigma^2(r_k)}{2}
f''(r_k)\right) \Bigg)\\
&&\qquad \qquad \qquad \qquad \qquad  \qquad \qquad 
- f(r_0)=\mathcal{O}(\eps^{2\ga+\zeta}).
\end{eqnarray*}
\end{lem}

\begin{proof}
 Let us denote
\begin{equation}\label{defeta}
\eta=f(r_{n_\ga})-\eps^2
\sum_{k=0}^{n_\ga-1}\left(b(r_k)f'(r_k)+\frac{ \sigma^2(r_k)}{2}
f''(r_k)\right).
\end{equation}
%
%
Writing,
\[f(r_{n_\ga})=f(r_0)+
\sum_{k=0}^{n_\ga-1}\left(f(r_{k+1})-f(r_k)\right)\]
and doing the Taylor expansion in each term inside the sum we get
\begin{eqnarray*}
f(r_{n_\ga})&=&f(r_0)+\sum_{k=0}^{n_\ga-1}\Big[
f'(r_k)(r_{k+1}-r_k)\\
&&+\frac{1}{2}f''(r_k)(r_{k+1}-r_k)^2+\mathcal{O}(\eps^3)\Big].
\end{eqnarray*}
Substituting this in \eqref{defeta} we get
\begin{equation}\label{eta-version2}
 \begin{split}
 \eta=&f(r_0)+\sum_{k=0}^{n_\ga-1}
\Big[f'(r_k)(r_{k+1}-r_k)+
\frac{1}{2}f''(r_k)(r_{k+1}-r_k)^2\Big]\\ 
&-\eps^2\sum_{k=0}^{n_\ga-1}
\left[b(r_k)f'(r_k)+\frac{\sigma^2(r_k)}{2}
f''(r_k)\right ] +\sum_{k=0}^{n_\ga-1}\mathcal{O}\left(\eps^3\right).
\end{split}
\end{equation}
Using \eqref{eq:NRmap-n} we can write
\[ 
\begin{split}
r_{k+1}-r_k&=\eps\omega_k[v(\theta_k,r_k)+\eps 
v_2(\theta_k,r_k)]+\eps^2E_2(\theta_k,r_k)
+\mathcal{O}(\eps^{2+a})\\
(r_{k+1}-r_k)^2&=\eps^2v^2(\theta_k,r_k)+\mathcal{O}(\eps^3).
\end{split}
\]
Thus, \eqref{eta-version2} can be written as
\begin{equation}\label{eta-version3}
 \begin{split}
  \eta=&f(r_0)+\eps\sum_{k=0}^{n_\ga-1}
f'(r_k)\omega_k\left[v(\theta_k,r_k)
+\eps v_2(\theta_k,r_k)\right]\\
 &+\eps^2\sum_{k=0}^{n_\ga-1}
f'(r_k)\left[E_2(\theta_k,r_k)-b(r_k)\right]
\\
 &+\frac{\eps^2}{2}\sum_{k=0}^{n_\ga-1}f''(r_k)
\left[v^2(\theta_k,r_k)-\sigma^2(r_k)\right]\\
 &+\sum_{k=0}^{n_\ga-1}\mathcal{O}(\eps^{2+a}).
\end{split}
 \end{equation}
Note first that since $\omega_k$ is independent of 
$(\theta_k,r_k)$  and $\E(\omega_k)=0$, we have 
\[
\begin{split}
\E(\omega_kf'(r_k)[v(\theta_k,r_k)+\eps v_2(\theta_k,r_k)])&=\\
\E(\omega_k)\E(f'(r_k)[v(\theta_k,r_k)+\eps v_2(\theta_k,r_k)])&=0.
\end{split}
\]
for all $k\in\mathbb{N}$. So, we do not need to analyze the term in the first 
row.

Using the law of total expectation and taking $\delta>0$ small enough, we 
split $\E(\eta)$ as 
\begin{equation}\label{def:TI:expectation:splittingexit}
\begin{split}
\E\left(\eta\right)=&\E\left(\eta\,|\,\eps^{-2(1-\ga)+\delta}\leq 
n_\ga\leq \eps^{-2(1-\ga)-\delta}\right)
\Prob\left\{\eps^{-2(1-\ga)+\delta}\leq n_\ga\leq 
\eps^{-2(1-\ga)-\delta}\right\}\\
+&\E\left(\eta\,|\,n_\ga< \eps^{-2(1-\ga)+\delta}\right)
\Prob\left\{n_\ga<\eps^{-2(1-\ga)+\delta}\right\}\\
+&\E\left(\eta\,|\,\eps^{-2(1-\ga)-\delta}<n_\ga\leq s\eps^{-2}\right)
\Prob\left\{\eps^{-2(1-\ga)-\delta}<n_\ga\leq s\eps^{-2}\right\}.
\end{split}
\end{equation}
We treat first the second and third rows. Taking into
account that 
\[
 \left|\E\left(\eta-f(r_0)\,|\,n_\ga< \eps^{-2(1-\ga)+\delta}\right)\right|
\leq K\eps^2n_\ga\leq
K\eps^{2\ga+\de}
\]
and using the first statement
of Lemma
\ref{lemma:TI:exittime}, we obtain the bound needed for the second row of
\eqref{def:TI:expectation:splittingexit}. For the third row, it is enough to 
use the second statement of Lemma \ref{lemma:TI:exittime} and 
\[
\left|\E\left(\eta-f(r_0)\,|\,
n_\ga\in\left(
\eps^{-2(1-\ga)-\delta},s\eps^{-2}\right]\right)\right|\leq K\eps^2n_\ga\leq
Ks.
\]

For the first row in \eqref{def:TI:expectation:splittingexit}, we need more 
accurate estimates. We need upperbounds for
\begin{equation}\label{def:As}
\begin{split}
A_1&=\eps^2\sum_{k=0}^{n_\ga-1}
f'(r_k)\left[E_2(\theta_k,r_k)-b(r_k)\right], \\
A_2&=\frac{\eps^2}{2}\sum_{k=0}^{n_\ga-1}f''(r_k)
\left[v^2(\theta_k,r_k)-\sigma^2(r_k)\right], \\
A_3&=\sum_{k=0}^{n_\ga-1}\mathcal{O}\left(\eps^{2+a}\right).
\end{split}
\end{equation}
with $\eps^{-2(1-\ga)+\delta}\leq 
n_\ga\leq \eps^{-2(1-\ga)-\delta}$.

For the last term  $A_3$, it is enough to use 
\begin{equation}\label{errortermsum1-small}
|A_3|\leq \left|\sum_{k=0}^{n_\ga-1}\mathcal{O}(\eps^{2+a})
\right|\leq K\eps^{2+a}n_\ga\leq K\eps^{2\ga+d}, 
\end{equation}
where $d=a-\delta>0$ due to smallness of $\delta$ and $K$ is independent of 
$\eps$.

The terms $A_1$ and $A_2$ are bounded analogously. We show how to bound the 
first one. Consider the constant $N$ given by Lemma \ref{lemmasigma2}.  Then, 
we write $n_\ga$ as $n_\ga=P_\ga N+Q_\ga$ for some $P_\ga$ and $0\leq 
Q_\ga<N$ and $A_1$ as $A_1=A_{11}+A_{12}$ with 
\[
 \begin{split}
A_{11}&= \eps^2\sum_{k=0}^{P_\ga-1}
\sum_{j=0}^{N-1}f'(r_{kN+j})\left[E_2(\theta_{kN+j},
r_{kN+j})-b(r_{kN+j})\right ], \\
A_{12}&=\eps^2\sum_{j=0}^{Q_\ga-1}
f'(r_{P_\ga N+j})\left[E_2(\theta_{P_\ga N+j},r_{P_\ga
N+j})-b(r_{P_\ga N+j})\right].
\end{split}
\]
The term $A_{12}$ can be bounded as $|A_{12}|\leq K\eps^{2}Q_\ga$. Now, by Lemma
\ref{lemmasigma2}, $Q_\ga<N\leq \eps^{-(\nu+b+2\tau)}$, which implies
\[
 |A_{12}|\leq K\eps^{2-\nu-b-2\tau}\leq K
\eps^{2\ga+\tau}\eps^{2(1-\ga)-\nu-b-3\tau}.
\]
Thus, it only suffices to check that $2(1-\ga)-\nu-3\tau\geq 0$. Using that
$\nu=1/4$, $\ga\in (4/5,4/5+1/40)$ and \eqref{conditionbeta}, we have
\[
 2(1-\ga)-\nu-b\geq \frac{1}{160}
\]
Therefore, taking $\tau\in (0,10^{-4})$, we have $2(1-\ga)-\nu-3\tau\geq 0$.

For the term $A_{11}$ we use \eqref{eq:NRmap-n} to obtain 
\[
\begin{split}
 A_{11}&= \eps^2\sum_{k=0}^{P_\ga-1}
\sum_{j=0}^{N-1}f'(r_{kN})\left[E_2(\theta_{kN}+jr_{kN},
r_{kN})-b(r_{kN})\right]+\OO(P_\ga N^3\eps^3) \\
&= \eps^2\sum_{k=0}^{P_\ga-1}
f'(r_{kN})\sum_{j=0}^{N-1}\left[E_2(\theta_{kN}+jr_{kN},
r_{kN})-b(r_{kN})\right]+\OO(P_\ga N^3\eps^3).
\end{split}
\]
Now, using Lemma \ref{lemmasigma2}, we have 
\[
 |A_{11}|\leq K\left(\eps^{2+\tau}P_\ga+P_\ga N^3\eps^3\right),
\]
for some constant $K>0$ independent of $\eps$. Using that 
$P_\ga$ and $N$ satisfy 
$$
P_\ga N\leq n_\ga\leq \eps^{-2(1-\ga)-\de}
$$ and $\ga\in (4/5,4/5+1/40)$, we have that 
\begin{equation}
|A_{11}|\leq K\left(\eps^{2\ga+\tau-\de}+\eps^{6\ga-3-3\de}\right)\leq
K\left(\eps^{2\ga+\tau-\de}+\eps^{2\ga+1/5-3\de}\right).
\end{equation}

Proceeding analogousy, one can bound $A_2$. Thus, it is enough to take
$\de<\tau,$  $\de<\frac{1}{15}$ and 
\[
 \zeta=\min\left\{\tau-\de,\frac{1}{5}-3\de ,a-\delta\right\}
\]
to obtain that, for $n\in (\eps^{-2(1-\ga)+\de},
\eps^{-2(1-\ga)-\de})$,
%
%
\[
\eta=f(r_0)+\eps\sum_{k=0}^{n_\ga-1}f'(r_k)\omega_k\left[
v(\theta_k,r_k)+\eps v_2(\theta_k,r_k)\right]+ \mathcal{O}(\eps^{2\ga+\zeta}).
\]
and therefore 
\[
\E\left(\eta\,|\,\eps^{-2(1-\ga)+\delta}\leq 
n_\ga\leq \eps^{-2(1-\ga)-\delta}\right) \times 
\]
\[
\Prob\{\eps^{-2(1-\ga)+\delta}\leq n_\ga\leq 
\eps^{-2(1-\ga)-\delta}\}=f(r_0)+ \mathcal{O}(\eps^{2\ga+\zeta}).
\]
This completes the proof of the lemma.
\end{proof}

\subsection{The IR case}\label{sec:IR-case}

The ideas to deal with Imaginary Rational strips are essentially the same as 
in the Totally Irrational case. Recall that  after performing the change
to normal form (Theorem \ref{thm:normal-form}), we are dealing with
\eqref{eq:NRmap-n}.
We also recall that given an imaginary rational strip $I_\ga$ there exists a
unique $r^*\in I_\ga$, with $r^*=p/q$ and $|q|<\eps^{-b}$, in its
$\eps^\nu$--neighborhood.
%
%

Fix an Imaginary Rational strip $I_\ga$ and 
Let $(\theta_0,r_0)\in I_\ga$. Recall that 
$n_\ga\leq s\eps^{-2}$ is either 
the exit time from $I_\ga$, that is the first number such that
$(\theta_{n_\ga+1},r_{n_\ga+1})\not\in I_\ga$ or the final time $n_\ga=n$. One 
has estimates for the exit time analogous to the ones in  Lemma 
\ref{lemma:TI:exittime}.

\begin{lem}\label{lemma:IR:exittime}

Fix $\ga\in (4/5,4/5+1/40)$. Then, there
exists a constant  $C>0$ such that, 
\begin{itemize}
\item For any $\de\in (0, 2(1-\ga))$ and $\eps>0$ small
enough,
\[
\Prob\{n_\ga<\eps^{-2(1-\ga)+\delta}\}\leq
e^{-{C\eps^{-\delta}}}.
\]
\item For any   $\de>0$ and $\eps>0$ small
enough,
\[
\Prob\{\eps^{-2(1-\ga)-\delta}<n_\ga<s\eps^{-2}\}\leq
e^{-{C\eps^{-\delta}}}.
\]
\end{itemize}
\end{lem}
\begin{proof}
We prove the second statement. The first one can be proved following the same 
lines as in Lemma \ref{lemma:TI:exittime} and the modifications that we use to 
prove the second statement. As in Lemma \ref{lemma:TI:exittime}, we 
define $\wt n_\ga=[\eps^{-2(1-\ga)}]$,
$n_\delta=[\eps^{-\delta}]$, and
$n_i=in_\ga$ and we use 
\[
\begin{split}
\Prob\left\{n_\ga>\eps^{-2(1-\ga)-\delta}\right\}\leq&\,\Prob\left\{|r_{n_{i+1}}
-r_ { n_i } |\leq
\eps^\ga\,\textrm{ for all }i=0,\dots, n_\delta-1\right\}\\
 \leq&\prod_{i=0}^{n_\delta}\Prob\left\{|r_{n_{i+1}}-r_{n_i}|\leq
\eps^{\ga}\right\}.
 \end{split}
\]
We have 
\begin{equation}\label{eqHni:IR}
r_{n_{i+1}}=r_{n_i}+\eps
\sum_{k=0}^{\wt n_\ga-1}\omega_k v\left(\theta_{n_i}+k 
r_{n_i},r_{n_i}\right)+\mathcal{O}(\wt n_\ga^3\eps^2). 
 \end{equation}
Considering $\xi$ defined in \eqref{def:xi}, we want to show that as $\wt 
n_\ga\to\infty$, it converges  in distribution to a normal 
random variable
$\mathcal{N}(0,\sigma^2(\theta_{n_i},r_{n_i}))$ with positive variance. 
Using Lemma \ref{main lemma}, we need a lower bound for 
\[
\sigma^2(\theta_{n_i},r_{n_i})=\lim_{\wt n_\ga\to\infty}
\frac{1}{\wt
n_\ga}\sum_{k=0}^{\wt n_\ga-1}v^2(\theta_{
n_i}+kr_{n_i},r_{n_i}).\]
Taking into account that there exists a rational $r=p/q$ with $d<q<\eps^{-b}$
in 
a $\eps^\nu$-neighborhood of  the imaginary rational strip 
$I_j$, we have
\[
\sigma^2(\theta_{n_i},r_{n_i})=\lim_{\wt n_\ga\to\infty}
\frac{1}{\wt
n_\ga}\sum_{k=0}^{\wt n_\ga-1}v^2(\theta_{
n_i}+kr_{n_i},p/q)+\OO\left(\eps^\nu\right).\]
The right hand side is a trigonometric polynomial of degree $2d$ in
$\theta_{n_i}$ and therefore it can have at most $4d$ zeros. Therefore, taking
$\eps$ small enough, we have that 
$\sigma^2(\theta_{n_i},r_{n_i})\geq K>0$ for some constant $K$. 
Then, the rest of the proof follows the same lines as in Lemma 
\ref{lemma:TI:exittime}.

 \end{proof}

\begin{lem}\label{lemmaexpectation-IR}
Let $\nu$,  $b=(\nu-\rho)/2$ and $\rho$ 
satisfy  \eqref{conditionbeta} and $\ga\in (4/5,4/5+1/40)$. Fix $\de>0$ 
small. Take 
$f:\mathbb{R}\rightarrow\mathbb{R}$ be any $\mathcal{C}^l$ function with $l\ge 
3$ and $\|f\|_{\CCC^3}\leq C$ for some constant $C>0$ independent of $\eps$. 
Then,  
\begin{eqnarray*}
 &&\E\left(f(r_{n_\ga})-\eps^2
\sum_{k=0}^{n_\ga-1}\left(b(r_k)f'(r_k)+\frac{\sigma^2(r_k)}{2
}
f''(r_k)\right)\right)\\
&&-f(r_0)=\mathcal{O}(\eps^{2\ga-\de}),
\end{eqnarray*}
where $b$ and $\sigma$ are the functions introduced in 
\eqref{eq:drift-variance}.
\end{lem}

\begin{proof}
Proceeding as in the proof of Lemma \ref{lemmaexpectation}, we define
\begin{equation}\label{defeta-IR}
\eta=f(r_{n_\ga})-
\eps^2
\sum_{k=0}^{n_\ga-1}\left(b(r_k)f'(r_k)+\frac{\sigma^2(r_k)}{2
}
f''(r_k)\right),
\end{equation}
which can be written as
\begin{equation}\label{eta-version3-IR}
 \begin{split}
 \eta=&f(r_0)+\sum_{k=0}^{n_\ga-1}f'(r_k)\eps\omega_k\left
[v(\theta_k,r_k)+\eps v_2(\theta_k,r_k)\right]\\
 &+\eps^2\sum_{k=0}^{n_\ga-1}f'(r_k)\left[E_2(\theta_k,
r_k)-b(r_k)\right]\\
 &+\frac{\eps^2}{2}\sum_{k=0}^{n_\ga-1}f''(r_k)\left[
v^2(\theta_k,r_k)-\sigma^2(r_k)\right]\\
 &+\sum_{k=0}^{n_\ga-1}\mathcal{O}(\eps^{2+a}).
\end{split}
 \end{equation}
Using the law of total expectation and taking $\de>0$ small enough,
\begin{eqnarray}\label{lawtotalexp-v2}
\beal 
\E\left(\eta\right)&=
\E\left(\eta\,|\,\eps^{-2(1-\ga)+\delta}\leq n_\ga\leq
\eps^{-2(1-\ga)-\delta}\right)
\Prob\left\{\eps^{-2(1-\ga)+\delta}\leq n_\ga\leq 
\eps^{-2(1-\beta)-\delta}\right\}\\
&+\E\left(\eta\,|\,n_\ga< \eps^{-(1-\ga)+\delta}\right)
\Prob\left\{n_\ga<
\eps^{-2(1-\ga)+\delta}\right\}\\
&+\E\left(\eta\,|\, \eps^{-2(1-\ga)-\delta}<n_\ga\leq s\eps^{-2}\right)
\Prob\left\{\eps^{-2(1-\ga)-\delta}<n_\ga\leq s\eps^{-2}\right\}.
\enal
\end{eqnarray}
By Lemma \ref{lemma:IR:exittime}, we have
\[
\begin{split}
\Prob\left\{n_\ga<\eps^{-2(1-\ga)+\delta}\right\}&\le 
e^{-{C\eps^{-\delta}}} \\
\Prob\left\{\eps^{-2(1-\ga)-\delta}<n_\ga\leq s\eps^{-2}\right\}&\le 
 e^{-{C\eps^{-\delta}}}.
\end{split}
\]
As in the proof of Lemma \ref{lemmaexpectation}, 
\begin{equation}\label{eq:IR:zeroexpeceps}
\begin{split}
\E(\omega_kf'(r_k)[v(\theta_k,r_k)+\eps v_2(\theta_k,r_k)])&=\\
\E(\omega_k)\E(f'(r_k)[v(\theta_k,r_k)+\eps v_2(\theta_k,r_k)])&=0.
\end{split}
\end{equation}
for all $k\in\mathbb{N}$ and  one can obtain the needed estimates for the
second and third row of \eqref{lawtotalexp-v2} exactly as in the proof of Lemma
\ref{lemmaexpectation}.

To upper bound the first row in \eqref{lawtotalexp-v2}, recall 
$ n_\ga\in (\eps^{-(1-\ga)+\delta},\eps^{-2(1-\ga)-\delta})$. Then, using
\eqref{eta-version3-IR} and  \eqref{eq:IR:zeroexpeceps},
\[
\left |\E\left(\eta-f(r_0)\,|\,\eps^{-2(1-\ga)+\delta}\leq n_\ga\leq
\eps^{-2(1-\ga)-\delta}\right)\right|\leq K\eps^2 n_\ga\leq K\eps^{2\ga-\delta}.
\]
This completes the proof of the lemma.
\end{proof}

\subsection{From a local diffusion to the global one: proof of Theorem 
\ref{maintheorem} }
\label{sec:IR-cyl-to-line}

In Sections \ref{sec:TI-case} and  \ref{sec:IR-case} we proved 
\emph{local} versions of formula \eqref{eq:suff-condition} in totally irrational and imaginary rational strips. 
Namely,  as long as we stay in one of the strips $I^j_\gamma$ of these two types, 
 for any $s>0$,  any time 
$n\le s\eps^{-2}$ and any $(\theta_0,r_0)$, as $\eps\to 0$, we have  
\begin{equation}\label{def:ExpNF}
\E(\eta_f)\to 0\,\, \text{ with }\eta_f= f(r_{n})- \eps^2
\sum_{k=0}^{n-1}\left(b(r_k)f'(r_k)+\frac{\sigma^2(r_k)}{2}
f''(r_k)\right)-f(r_0).
\end{equation}
In \cite{CGK}, an analogous analysis is done for the resonant strips. To complete the proof of Theorem 
\ref{maintheorem}, it suffices to  prove the global version in the whole 
cylinder. Namely, when the iterates visit  totally irrational, 
imaginary rational strips  and resonant zones.

To this end, we need to analyze how the iterates visit the different strips. We
 model these visits as a \emph{random walk}. It turns out that in the core of resonant zones we face serious technical difficulties since they are significantly different from  the non-resonant zones  (see
\cite{CGK}). Since the cores have a very
small measure, we prove that the fraction of time spent in those cores is
rather low and, thus, has small influence in the long time behavior.

To be  able to finally combine the resonant and non-resonant regimes, we
consider a second division of both the resonant and non-resonant zones in strips
of bigger size than $I_\ga^j$. The behavior in those strips will be the same at
either non-resonant and resonant strips. This will allow us to later ``join''
both regimes.

Fix a parameter $\kk\in (1/3,1/11)$ and divide both resonant and
non-resonant zones  into intervals $\I_\kk^j$ of length $\eps^\kk$. The
non-resonant zones are chosen so that   the endpoints of those strips coincide
with endpoints of the previous grid of strips $I_\ga^j$. Each interval
$\I_\kk^j$ contains $\eps^{\kk-\ga}$ $I_\ga^j$ strips. This new division at the
resonant zones is done in \cite{CGK}.

We prove in the non-resonant strips $\I_\kk^j$ a result analogous to Lemma
\ref{lemmaexpectation}. Namely, we show that, since the relative measure of
Imaginary Rational strips is very small, the behavior in the strip  $\I_\kk^j$
is given by the behavior of the Totally Irrational substrips $I_\gamma^j$.

\begin{lem}\label{lemma:expectlemmaBigStrips}
Consider $C>0$,  $\kk\in (1/3,1/11)$ and a strip $\I_\kk^j$ in the
non-resonant zone $\mathcal D_\beta$ (see \eqref{eq:non-res-domain}). Take 
$f:\mathbb{R}\rightarrow\mathbb{R}$ be any $\mathcal{C}^l$ function with $l\ge 
3$ and $\|f\|_{\CCC^3}\leq C$. Then there exists $\zeta>0$ such that
\begin{equation}\label{eq:ExpectIntermStrips}
\E\Bigg(f(r_{n_\kk})- \eps^2
\sum_{k=0}^{n_\kk-1}\left(b(r_k)f'(r_k)+\frac{\sigma^2(r_k)}{2}
f''(r_k)\right) \Bigg)
- f(r_0)=\mathcal{O}(\eps^{2\kk+\zeta}).
\end{equation}
where $b$ and $\sigma$ are the functions defined in \eqref{eq:drift-variance}.

Moreover, call $n_\kk$ the exit time from these strips. Then,  there
exists a constant  $C'>0$ such that, 
\begin{itemize}
\item For any $\de\in (0, 2(1-\kk))$ and $\eps>0$ small
enough,
\[
\Prob\{n_\kk<\eps^{-2(1-\kk)+\delta}\}\leq
e^{- C' \eps^{- \delta}}.
\]
\item For any   $\de>0$ and $\eps>0$ small
enough,
\[
\Prob\{\eps^{-2(1-\kk)-\delta}<n_\ga<s\eps^{-2}\}\leq
e^{- C' \eps^{- \delta}}.
\]
\end{itemize}
\end{lem}

This lemma is proven in Section \ref{sec:ProofExpecIntermediateStrips}. An analogous lemma for the resonant zones is proven in \cite{CGK}. In that lemma we replace $\mathcal D_\beta$ from \eqref{eq:non-res-domain}
by $\mathcal R_\beta^{p/q}$ from \eqref{eq:res-domain} and  the $r$-component by the Hamiltonian $H$. All the rest is the same.

%
%

\subsubsection{Proof of Lemma \ref{lemma:expectlemmaBigStrips}}\label{sec:ProofExpecIntermediateStrips}
The strip $\I_\kk=\I_\kk^j$ is the union of $\eps^{\kk-\ga}$ totally irrational
and imaginary rational strips. We  analyze the amount of visits that are done to
each strip and we
prove that the time spent in Imaginary Rational strips is small compared with 
the time spent in the Totally Irrational strips. Assume $r_0=0$ (if not just 
apply a translation). We want to 
model the visits to the different strips in $\I_\kk$ by a symmetric random walk.  

Modifying slightly the strips considered in 
Sections 
\ref{sec:TI-case} and \ref{sec:IR-case}, we consider endpoints of the strips 
\[
 r_j=A_j \eps^\ga,\quad j\in\ZZ
\]
with some (later determined) constants $A_j$ independent of $\eps$ to the leading order and satisfying $A_0=0$, $A_1=A>0$ 
and 
$A_j<A_{j+1}$ for $j>0$ (and similarly for $j<0$).  We consider the strips 
\[
 I_\ga^j=[r_j,r_{j+1}]=[A_j\,\eps^\ga, A_{j+1}\,\eps^\ga].
\]
To analyze the visits to these strips, we consider the lattice of points 
$\{r_j\}_{j\in\ZZ}\subset\R$ and we analyze the ``visits'' to these points. By 
visit we mean the existence of an iterate $\OO(\eps)$-close to it.  Lemmas 
\ref{lemma:TI:exittime} and \ref{lemma:IR:exittime} imply that 
if we start with $r=r_j$ we hit either $r_{j-1}$ or 
$r_{j+1}$  with probability one. This process  can be treated as a random walk 
for $j\in\ZZ$,
\begin{equation}\label{def:randomwalk:inter}
 S_j=\sum_{i=0}^{j-1}Z_i,
\end{equation}
where $Z_i
$ are Bernouilli variables taking values $\pm 1$. $Z_i$'s are not necessarily
symmetric.  Thus, we choose 
the constants $A_j>0$ so that the $Z_i$ are Bernouilli 
variables with $p=1/2$.

\begin{lem}\label{lemma:randomwalkinter}
There exist constants  $J^\pm>0$  independent of $\eps$ 
and  $\{A_j\}_j,\  j\in \left[\lfloor J_-\eps^{\kk-\ga}\rfloor, \lfloor 
J_+\eps^{\kk-\ga}\rfloor \right]$ such that 
\begin{itemize}
 \item
$A_j=A_{j-1}+(A_1-A_{0})\exp(-\int_0^{r_{j-1}}\frac { 
2b(r)}{\sigma^2(r)}dr)+\OO(\eps^{\ga}).$

\item $\displaystyle \I_\kk\subset\bigcup_{j=\lfloor 
J_-\eps^{\kk-\ga}\rfloor}^{\lfloor 
J_+\eps^{\kk-\ga}\rfloor}[A_j,A_{j+1}]$.
\item The random walk process induced by the map \eqref{eq:NRmap-n} on the 
lattice $\displaystyle\{r_j\}_j,\  
j\in \left[ \lfloor J_-\eps^{\kk-\ga}\rfloor, \lfloor 
J_+\eps^{\kk-\ga}\rfloor \right]$ is a symmetric random walk.
\end{itemize}
\end{lem}

%


\begin{proof}
To compute the probability of hitting (an $\eps$-neighborhood of) either 
$r_{j\pm 1}$ from $r_j$, we use the local 
expectation lemmas (Lemmas \ref{lemmaexpectation} and 
\ref{lemmaexpectation-IR}). Therefore we can consider $f$ in the kernel of the 
infinitesimal generator $A$ of the diffusion process (see 
\eqref{eq:diffusion-generator}) and  solve the boundary problem
\[
 b(r)f'(r)+\frac{1}{2}\sigma^2(r)f''(r)=0,\qquad f(r_{j-1})=0,\quad f(r_{j+1})=1
\]
The solution gives the probability of hitting $r_{j+1}$ before hitting 
$r_{j-1}$  starting at a given 
$r\in[r_{j-1},r_{j+1}]$. The unique solution is given by
\[
 f(r)=\frac{\int_{r}^{r_{j+1}}\exp(-\int_0^\rr 
\frac{2b(s)}{\sigma^2(s)}ds)\ d\rr}{
\int_{r_{j-1}}^{r_{j+1}}\exp(-\int_0^\rr 
\frac{2b(s)}{\sigma^2(s)}ds)\ d\rr}.
\]
We use $f$ to choose the coefficients $A_j$ iteratively (both as $j>0$ increases 
and $j<0$ decreases). Assume that $A_{j-1}$, $A_j$ have been fixed. Then, to 
have a symmetric random walk, we have to choose $A_{j+1}$ such that 
$f(r_j)=1/2$. 

Define
\[
  m(r)=\exp\left(-\int_0^\rr 
\frac{2b(s)}{\sigma^2(s)}ds\right) 
 \]
 and $D_j=A_j-A_{j-1}$. Then, using the mean value theorem, $f(r_j)=1/2$ can be 
 written as 
 \[
  \frac{m(\xi_j)D_j}{m(\xi_j)D_j+m(\xi_{j+1})D_{j+1}}=\frac{1}{2}
 \]
 where $\xi_j\in [A_{j-1},A_j]$ and $\xi_{j+1}\in [A_{j},A_{j+1}]$. Thus, one 
has
 \[
  D_{j+1}=\frac{m(\xi_{j})}{m(\xi_{j+1})}D_j \quad \text{  which implies }\quad 
  D_{j+1}=\frac{m(\xi_{1})}{m(\xi_{j+1})}D_1.
 \]
Thus the length $D_j$ of the strip $I_\ga^j=[r_j,r_{j+1}]=[A_j\eps^\ga, 
A_j\eps^\ga]$.
 \[
  D_j=\frac{m(\xi_{1})}{m(\xi_{j})}\ D_0
 =A  \exp\left(-\int_0^{r_{j-1}} 
\frac{2b(s)}{\sigma^2(s)}ds)+\OO(\eps^{\ga}\right).
 \]
%
%
The distortion of the strips does not depend on $\eps$ (at first order). Therefore, adjusting $A$ and $J_+$ one can obtain the intervals 
$[r_{j}, r_{j+1}]=[A_{j}\eps^\ga, A_{j+1}\eps^\ga]$ which  cover $\I_\kk$ with $r>0$. 
Proceeding analogously for $j<0$, one can do the same for $\I_\kk$ with $\{r<0\}$.
\end{proof}

To prove \eqref{eq:ExpectIntermStrips}, we need to combine the iterations within each strip $I_\ga^j$ and the random walk evolution among the strips. Since we have $\eps^{\kk-\ga}$ strips, the exit time $j^*$ for the random walk $S_j$ from $\I_\kk$ satisfies the following. There exists $C>0$ such that for any small $\de$ and $\eps$,
\begin{equation}\label{eq:Probj*}
 \begin{split}
 \Prob \left(j^*\geq \eps^{2(\kk-\ga)-\frac{\de}{2}}\right) & \leq
e^{-C\, \eps^{-{\de/2}}}\\
 \Prob \left(j^*\leq \eps^{2(\kk-\ga)+\frac{\de}{2}}\right) & \leq
e^{-C\, \eps^{-{\de/2}}}.
 \end{split}
\end{equation}
We use this to obtain the probabilities for the exit time $n_\kk$ stated in
Lemma \ref{lemma:expectlemmaBigStrips}. We prove the second statement for
$n_\kk$, the other one can be proved analogously. Call $j^*$ the exit time for
the random walk and $n_\ga^j$, $j=1, \ldots,j^*$ the exit times for the $j^*$
visited strip before hitting the endpoints of $\I_\kk$. Define also
$\Delta_j=n_\ga^j-n_\ga^{j-1}$ with $j\geq 2$, $\Delta_1=n_\ga^1$ and 
$X=\{\eps^{-2(1-\kk)-\delta}<n_\kk<s\eps^{-2}\}$. We condition the
probability as follows,
{\small \[
\begin{split}
&\Prob\{X\}\\
&\leq\Prob\left\{X\left|j^*\in
(\eps^{2(\kk-\ga)+\frac{\de}{2}},\eps^{2(\kk-\ga)-\frac{\de}{2}}), \Delta_j\in(
\eps^{-2(1-\ga)+\frac{\de}{2}},\eps^{-2(1-\ga)+\frac{\de}{2}}), j=1,\ldots,
j^*\right.\right\}\\
&\times \Prob\left\{j^*\in (\eps^{2(\kk-\ga)+\frac{\de}{2}},\eps^{2(\kk-\ga)-\frac{\de}{2}}), \Delta_j\in( \eps^{-2(1-\ga)+\frac{\de}{2}},\eps^{-2(1-\ga)+\frac{\de}{2}}), j=1,\ldots, j^*\right\}\\
&+\Prob\left\{X\left|j^*\not\in
(\eps^{2(\kk-\ga)+\frac{\de}{2}},\eps^{2(\kk-\ga)-\frac{\de}{2}}) \text{ or }
\exists j, \Delta_j\not\in(
\eps^{-2(1-\ga)+\frac{\de}{2}},\eps^{-2(1-\ga)+\frac{\de}{2}})\right.\right\}\\
&\times\Prob\left\{j^*\not\in
(\eps^{2(\kk-\ga)+\frac{\de}{2}},\eps^{2(\kk-\ga)-\frac{\de}{2}}) \text{ or }
\exists j, \Delta_j\not\in(
\eps^{-2(1-\ga)+\frac{\de}{2}},\eps^{-2(1-\ga)+\frac{\de}{2}})\right\}
\end{split}
\]}
For the first term in the conditionned probability we show that 
{\small \[
 \Prob\left\{X\left|j^*\in
(\eps^{2(\kk-\ga)+\frac{\de}{2}},\eps^{2(\kk-\ga)-\frac{\de}{2}}), \Delta_j\in(
\eps^{-2(1-\ga)+\frac{\de}{2}},\eps^{-2(1-\ga)+\frac{\de}{2}}), j=1,\ldots,
j^*\right.\right\}=0
\]}
Indeed, we have that 
\[
 n_\kk=\sum_{j=1}^{j^*}n_\ga^j\leq j^*\sup_j n_\ga^j<
\eps^{2(\kk-\ga)-\frac{\de}{2}}\cdot
\eps^{-2(1-\ga)-\frac{\de}{2}}\leq\eps^{-2(1-\kk)-\de}.
\]
Therefore, we only need to bound the second term in the conditioned probability. To this end, we need 
an upper bound for the number of visited strips. Since $n\le s\eps^{-2}$ and $|r_n-r_{n-1}|\lesssim\eps$, there exists a 
constant $c>0$ such that 
\[
\Delta_j=n_\ga^j-n_\ga^{j-1}\geq c\eps^{\ga-1}\quad\text{ for }j=0,\ldots, j^*-1.
\]
This implies that 
\begin{equation}\label{def:maxvisitedstrips}
j^*\lesssim \eps^{-1-\ga}.
\end{equation}
Thus, using Lemmas  \ref{lemma:TI:exittime} and \ref{lemma:IR:exittime},
{\small 
\[
\begin{split}
&\Prob\left\{j^*\not\in (\eps^{2(\kk-\ga)+\frac{\de}{2}},\eps^{2(\kk-\ga)-\frac{\de}{2}}) \text{ or } \Delta_j\not\in( \eps^{-2(1-\ga)+\frac{\de}{2}},\eps^{-2(1-\ga)+\frac{\de}{2}}) \text{ for some } j=1,\ldots, j^*\right\}\\
&\leq\eps^{-1-\ga}  e^{-C\eps^{-\de/2}}
\end{split}
\]}
Thus, taking a smaller $C>0$ and taking $\eps$ small, we obtain the second
statement for $n_\kk$ in Lemma \ref{lemma:expectlemmaBigStrips}. One can prove
the lower bound for $n_\kk$ analogously.

It only remains to prove \eqref{eq:ExpectIntermStrips}.
We define the 
Markov times  $0=n_\ga^0<n_\ga^1<n_\ga^2 <\dots <n_\ga^{j^*-1}<n_\ga^{j^*}<n$
for some random $j^*=j^*(\om)$ such that each $n_\ga^j$ is the stopping 
time as in \eqref{eq:stopping-time}, where $j^*$ denotes either the exit time
from $\I_\kk$ or the last change between strips $I_\ga^j$ inside $\I_\kk$.
By \eqref{eq:Probj*},  $j^*(\om)$ is the exit time except for an exponentially
small probability. We use conditionned expectation
as 
\[
 \E(\eta_f)=\E(\eta_f|A_1)\Prob (A_1)+\E(\eta_f|A_2)\Prob(A_2)
\]
with 
\[
 \begin{split}
  A_1=&\Big\{\eps^{-2(1-\kk)-\delta}<n_\kk<\eps^{-2(1-\kk)+\delta}, j^*\in
(\eps^{2(\kk-\ga)+\frac{\de}{2}},\eps^{2(\kk-\ga)-\frac{\de}{2}}),\\
&\Delta_j\in( \eps^{-2(1-\ga)+\frac{\de}{2}},\eps^{-2(1-\ga)+\frac{\de}{2}}),
j=1,\ldots, j^*\Big\}\\
A_2=&A_1^c.
 \end{split}
\]
Lemmas \ref{lemma:TI:exittime}, \ref{lemma:IR:exittime}, the estimates for
$n_\kk$ given in Lemma \ref{lemma:expectlemmaBigStrips} and 
\eqref{eq:Probj*} imply that $\Prob(A_2)\ll \eps^{2\kk+\zeta}$. Moreover, since
we only  consider functions $f$ such that $\|f\|_{\mathcal 
C^3}\leq C$ with $C>0$ independent of $\eps$, we have that 
\[
\left|\E(\eta_f|A_2)\Prob(A_2)\right|\lesssim \eps^{2\kk+\zeta}.
\]
Therefore, it only remains to bound $\E(\eta_f|A_1)\Prob (A_1)$. We use that
$\Prob (A_1)\leq 1$ and we estimate $\E(\eta_f|A_1)$.

We decompose the above sum as
$\eta_f=\sum_{j=0}^{j^*} \eta_j$ with 
\begin{eqnarray*}
\eta_j=&&f(r_{n_\ga^{j+1}})-
f(r_{n_\ga^{j}})- \qquad \qquad \\
&&  \eps^2
\sum_{s=n_\ga^j}^{n_\ga^{j+1}}\left(b(r_s)f'(r_s)+\frac{\sigma^2(r_s)}{2}
f''(r_s)\right).
\end{eqnarray*}
Theorems \ref{lemmaexpectation} and \ref{lemmaexpectation-IR} imply that for 
any $j$, 
\begin{equation}\label{def:localExpec:Intermediate}
\begin{split}
|\E(\eta_j)|&\lesssim\eps^{2\ga+\zeta}\qquad\text{ for totally irrational 
strips}\\
|\E(\eta_j)|&\lesssim\eps^{2\ga-\de}\qquad\quad\text{for imaginary rational 
strips},
\end{split}
\end{equation}
for some $\de>0$ arbitrarily small and some $\zeta>0$. To use these estimates,
we
need to control how many visits we do to each type of strips. Taking into
account that the visits to the strips are modelled by the symmetric random walk
$S_j$. Denote by 
$B\subset M= \{1,\ldots,\lceil \eps^{\kk-\ga}\rceil\}\subset\mathbb N$ the
endpoints of the Imaginary rational strips $I_\ga^j$ in $\I_\kk$.  By Appendix
\ref{sec:measure-IR-RR}, we know that 
\[
 |B|\lesssim \eps^{\kk-\ga+\rho}.
\]
Denote by $\mu=  |B|/\lceil \eps^{\kk-\ga}\rceil$ the relative measure of $B$ in $M$.
\begin{lem}\label{lemma:ShortVisitsB} Fix $\de>0$ small. There exists a constant
$C>0$ such that for $\eps>0$ small enough,
\[
 \Prob\left(\sharp \left\{j\in [0,j^*):S_j\in B\right\}\geq
j^*\mu\eps^{-\de} \right)\leq e^{-C \eps^{-\de/2}}
\]
\end{lem}
\begin{proof}
We have that
\[
 \Prob\left(\sharp \left\{j\in [0,j^*):S_j\in B\right\}\geq
j^*\mu\eps^{-\de} \right)= \Prob\left(\sum_{k\in B}\sharp \left\{j\in [0,j^*):S_j=k\right\}\geq
j^*\mu\eps^{-\de} \right)
\]
Take any $k^*\in B$, then
\[
 \Prob\left(\sum_{k\in B}\sharp \left\{j\in [0,j^*):S_j=k\right\}\geq
j^*\mu\eps^{-\de} \right)\leq  \Prob\left(\sharp \left\{j\in [0,j^*):S_j=k^*\right\}\geq
j^*\eps^{-\ga+\kk-\de} \right).
\]
Since we start the random walk at $S_0=0$, it is clear that the probability of visiting $k^*$ $j$-times is lower than the probability of visit $0$ $j$-times. Namely, 
\[
 \Prob\left(\sharp \left\{j\in [0,j^*):S_j=k^*\right\}\geq
j^*\eps^{-\ga+\kk-\de} \right)\leq \Prob\left(\sharp \left\{j\in [0,j^*):S_j=0\right\}\geq
j^*\eps^{-\ga+\kk-\de} \right)
\]
We prove that such probability is exponentially small in $\eps$. Denote by $f_k$
the random variable that gives the number of iterates between the $k-1$ and $k$
visiting zero. Then, 
\[
\begin{split}
  \Prob\left(\sharp \left\{j\in [0,j^*):S_j=0\right\}\geq
j^*\eps^{-\ga+\kk-\de} \right) &=\Prob\left(\sum_{k=1}^{\lceil j^*\eps^{-\ga+\kk-\de} \rceil}f_k\leq j^*\right)\\
&\leq \prod_{k=1}^{\lceil j^*\eps^{-\ga+\kk-\de} \rceil}\Prob\left(f_k\leq j^*\right).
\end{split}
\]
Since the random variables $\{f_k\}$ are independent identically distributed,
\[
  \Prob\left(\sharp \left\{j\in [0,j^*):S_j=0\right\}\geq
j^*\eps^{-\ga+\kk-\de} \right)\leq \Prob\left(f_1\leq j^*\right)^{\lceil j^*\eps^{-\ga+\kk-\de} \rceil}.
\]
Since we are dealing with a symmetric random walk, it is well known that
\[
 \Prob(f_1= m)=\begin{pmatrix} 2m\\ m\end{pmatrix}\frac{2^{-2m}}{2m-1}.
\]
which satisfies
\[
 \Prob(f_1= m)\sim \frac{1}{\sqrt{\pi m}(2m-1)} \qquad\text{ as }m\to +\infty.
\]
Therefore, there exists a constant $c>0$ such that for $m$ large enough
\[
 \Prob(f_1\leq m)\leq 1-cm^{-1/2}\]
Then, one can conclude that 
\[
 \begin{split}
 \Prob\left(\sharp \left\{j\in [0,j^*):S_j\in B\right\}\geq
j^*\mu\eps^{-\de} \right)&\geq  \Prob\left(f_1\leq j^*\right)^{\lceil j^*\eps^{-\ga+\kk-\de} \rceil}\\
&   \leq \left(1- \frac{c}{(j^*)^{1/2}}\right)^{\lceil j^*\eps^{-\ga+\kk-\de} \rceil}\\
 &  \leq e^{-C \eps^{-{\de/2}}},
 \end{split}
\]
for some constant $C>0$ independent of $\eps$ and  $\eps$ small enough.
%
%
\end{proof}

Lemma \ref{lemma:ShortVisitsB} implies that it is enough to deal
with the case 
\[
\{n:S_n\in B\}\leq j^*\mu\eps^{-\de}\leq
\eps^{2(\kk-\ga)+\rho-2\de},
\]
where we have used that $j^*\leq\eps^{2(\kk-\ga)-\de}$ and $\mu\leq \eps^\rho$.
Using this and  \eqref{def:localExpec:Intermediate}, we can deduce that
\[
\begin{split}
\left|\E\left(\left.\sum_{j=0}^{j^*}\eta_j\right| A_1\right)\right|
&\lesssim
\eps^{2\ga+\zeta}\eps^{-2(\ga-\kk)-2\de}
+\eps^{2\ga-\de}\eps^{-2(\ga-\kk)+\rho-2\de}\\
&\leq \eps^{2\kk+\zeta-2\de}+\eps^{2+\kk+\rho-3\de}.
\end{split}
\]
Therefore, taking $\de>0$ small enough, we have proven
\eqref{eq:ExpectIntermStrips}.

\subsubsection{Proof of Theorem \ref{maintheorem}}
To complete the  proof of Theorem \ref{maintheorem} it is enough to use Lemmas
\ref{lemma:expectlemmaBigStrips} and the
corresponding lemma for the resonant strips
given in \cite{CGK} and
model the visits to the strips $\I_\kk^j$ as a random walk as we have done for
the strips $I_\ga^j$ to prove Lemma \ref{lemma:expectlemmaBigStrips} in Section
\ref{sec:ProofExpecIntermediateStrips}.

This proof is slightly different since we are dealing with a
non-compact domain and therefore we need estimates for the low probability of
doing big excursions. As before,  we assume $r_0=0$ (if not just 
apply a translation) and we 
treat the visits to the different strips $\I_\kk^j$ by a  random walk. 
Consider  $R\gg 1$, which we will fix a posteriori, and
consider the 
endpoints of the strips $[-R,R]$.

To prove \eqref{def:ExpNF}, we condition the expectation in a different way as
for the proof of Lemma \ref{lemma:expectlemmaBigStrips}. We condition it  as 
\begin{equation}\label{def:SplittingExpect}
\begin{split}
 \E (\eta)= &\E \left(\eta \left||r_n|< R\,\text{ for all }n\leq 
s\eps^{-2}\right.\right)\Prob\left(|r_n|< R\,\text{ for all }n\leq 
s\eps^{-2}\right)\\
&+ \E \left(\eta \left|\exists n^*\leq 
s\eps^{-2} \text{ with }|r_n|\geq R\right.\right)\Prob\left(\exists n^*\leq 
s\eps^{-2} \text{ with }|r_n|\geq R\right).
\end{split}
\end{equation}
We bound each row. We start with the second one. 

Since we are considering 
$n\leq 
s\eps^{-2}$ and we consider functions $f$ such that $\|f\|_{\mathcal 
C^3(\mathbb R)}\leq C$ with $C>0$ independent of $\eps$, we have that 
\[
\left| \E \left(\eta \left|\exists n^*\leq 
s\eps^{-2} \text{ with }|r_n|\geq R\right.\right)\right|\leq C'
\]
for some $C'>0$ which depends on $s$ but is independent of $\eps$ and $R$.
Thus, 
to bound the second row, it is 
enough to prove that choosing $R$ large enough, $\Prob\left(\exists n^*\leq 
s\eps^{-2} \text{ with }|r_{n^*}|\geq R\right)$ can be made as small as desired 
uniformly for small $\eps$.

We divide the interval $[-R,R]$ into equal substrips $\I_\kk^j$ of length equal
to $\eps^\kk$. It is clear that there are $R\eps^{-\kk}$ strip. We model the
visits to these strips as a non-symmetric random walk $S_j$ 
in \eqref{def:randomwalk}. Note that the this is significantly different from
Section \ref{sec:ProofExpecIntermediateStrips} since now the probabilities of
going left or right depend on the point (because of the drift). 

Note that now the random walk $S_j=\sum_{k=1}^j Z_j$ where each $Z_j$  is a
Bernouilli variable with probabilities $p_j$, $q_j$ which depend on the visited
strip. Proceeding as in the proof of Lemma \ref{lemma:randomwalkinter} and
taking into account that we have uniform bounds for the drift given in Theorem
\ref{thm:normal-form}, one can prove that at every strip the probabilities
$p_j$, $q_j$ satisfy
\[
 \left|p_j-\frac{1}{2}\right|\leq C\eps^\kk,\quad \left|q_j-\frac{1}{2}\right|\leq C\eps^\kk
\]
for some constant $C>0$ which is independent of $\eps$ and $R$. As a consequence,
\begin{equation}\label{eq:ExpecBernouilli}
 \left|\E Z_j\right| \leq 2C\eps^\kk.
\end{equation}

Call $j^*$ the first visit to one of the strips 
containing $r=\pm R$.  It is clear that 
\[
 j^*\geq R\eps^{-\kk}.
\]
We fix $\de>0$ small and we condition $\Prob\left(|r_n|< R\,\text{ for all }n\leq 
s\eps^{-2}\right)$ as follows. Call $X=\{|r_n|< R,\text{ for all }n\leq 
s\eps^{-2}\}$,
\begin{equation}\label{eq:ExitProbFinal}
 \begin{split}
  \Prob(X)=&\Prob\left(X\left|j^*\leq R^\de\eps^{-2\kk}\right.\right)\Prob\left(j^*\leq R^\de\eps^{-2\kk}\right)\\
  &+\Prob\left(X\left|j^*> R^\de\eps^{-2\kk}\right.\right)\Prob\left(j^*> R^\de\eps^{-2\kk}\right).
 \end{split}
\end{equation}
For the first row it is enough to use $|\Prob\left(X\left|j^*\leq R^\de\eps^{-2\kk}\right.\right)|\leq 1$ and the following lemma.

\begin{lem}
Fix $\eps_0>0$. Then,   for any $\eps\in (0,\eps_0)$ and $R>0$ large enough, 
\[
 \Prob\left(j^*\leq R^\de\eps^{-2\kk}\right)\leq e^{-CR^{2-\de}}
\]
for some constant $C>0$ independent of $\eps$ and $C>0$.
\end{lem}
\begin{proof}
Since the number of strips is $R\eps^{-\kk}$,
\[
 \Prob\left(j^*\leq R^\de\eps^{-2\kk}\right)\leq  \Prob\left(\exists j^*\leq
R^\de\eps^{-2\kk}:\left|\sum_{k=1}^{j^*}Z_j\right|\geq R\eps^{-\kk}\right).
\]
Define $Y_j=Z_j-\E Z_j$, then for $R$ large enough and taking \eqref{eq:ExpecBernouilli} into account
\[
\begin{split}
 \Prob\left(\left|\sum_{k=1}^{j^*}Z_j\right|\geq R\eps^{-\kk}\right)= &\leq  \Prob\left(\left|\sum_{k=1}^{j^*}Y_j+\sum_{k=1}^{j^*}\E Z_j\right|\geq R\eps^{-\kk}\right)\\
 &\leq   \Prob\left(\left|\sum_{k=1}^{j^*}Y_j\right|\geq R\eps^{-\kk}-C j^*\eps^{\kk}\right)\\
 &=   \Prob\left(\left|\frac{1}{\sqrt{j^*}}\sum_{k=1}^{j^*}Y_j\right|\geq
\frac{R\eps^{-\kk}}{\sqrt{j^*}}-C \sqrt{j^*}\eps^{\kk}\right).
\end{split}
\]
Using that $j^*\leq R^\de\eps^{-2\kk}$, taking $R$ big enough,
\[
 \frac{R\eps^{-\kk}}{\sqrt{j^*}}-C \sqrt{j^*}\eps^{\kk}\leq   \frac{R\eps^{-\kk}}{2\sqrt{j^*}}
\]
which implies,
\[
  \Prob\left(\left|\sum_{k=1}^{j^*}Z_j\right|\geq R\eps^{-\kk}\right)\leq 
\Prob\left(\left|\frac{1}{\sqrt{j^*}}\sum_{k=1}^{j^*}Y_j\right|\geq
\frac{R\eps^{-\kk}}{2\sqrt{j^*}}\right).
\]
The variables $Y_j$ are independent but not identically distributed.
Nevertheless,  their third moments have a uniform upper  bound independent of 
$\eps$ and $R$. Then, one can apply Lyapunov center limit theorem to prove that 
\[\frac{1}{\sqrt{j^*}}\sum_{k=1}^{j^*}Y_j\]
tends in distribution  to a normal random variable with
positive variance which has a lower bound independent of $\eps$ and $R$.
Therefore, 
\[
  \Prob\left(\left|\sum_{k=1}^{j^*}Z_j\right|\geq R\eps^{-\kk}\right)\leq  e^{-C' \frac{R^2\eps^{-2\kk}}{4j^*}},
\]
for some $C'>0$ independent of $\eps$ and $R$. This implies that 
\[
  \Prob\left(\exists j^*\leq R^\de\eps^{-2\kk}:\left|\sum_{k=1}^{j^*}Z_j\right|\geq R\eps^{-\kk}\right)\leq  e^{-C'R^{2-\de}}.
\]
reducing slightly $C'$ if necessary.
\end{proof}

Now we bound the second row in \eqref{eq:ExitProbFinal}.
Call $N_j$ the exit time for 
$r_n$ of the
$j$-th visit. The expectation $\E 
N_j$ depends on the visited strip but is independent of  $j$ since the different
visits to the same strip are independent. 
Moreover, a direct consequence of Lemmas \ref{lemma:expectlemmaBigStrips} and
the analogous lemma for resonant zones given in \cite{CGK} is 
that 
\[
 C^{-1}\eps^{-2(1-\kk)}\leq \E N_j\leq C\eps^{-2(1-\kk)}
\]
for some constant $C>0$  independent of $\eps$ and $R$ (the lengths of the
strips are $R$ independent). 

To bound the first row in \eqref{eq:ExitProbFinal}, we use $\Prob\left(j^*\leq
R^\de\eps^{-2\kk}\right)\leq 1$ and we condition $\Prob\left(X\left|j^*\leq
R^\de\eps^{-2\kk}\right.\right)$ as follows. Fix $\lambda>0$ small independent
of $\eps$ and $R$.
\begin{equation}\label{eq:ProbExitFinal2}
\begin{split}
 \Prob\left(X\left|j^*\leq
R^\de\eps^{-2\kk}\right.\right)=& \Prob\left(X\left|j^*\leq
R^\de\eps^{-2\kk},\left. \left|\frac{1}{j^*}\sum_{j=1}^{j^*}M_j\right|
>\lambda\right.\right.\right)\Prob\left(\left|\frac{1}{j^*}\sum_{j=1}^{j^*}
M_j\right| >\lambda\right)\\
&+\Prob\left(X\left|j^*\leq
R^\de\eps^{-2\kk},\left. \left|\frac{1}{j^*}\sum_{j=1}^{j^*}M_j\right|
\leq\lambda\right.\right.\right)\Prob\left(\left|\frac{1}{j^*}\sum_{j=1}^{j^*}
M_j\right| \leq\lambda\right)
\end{split}
\end{equation}
We start by bounding the first row. Define the variables
\[
 M_j=\frac{N_j-\E N_j}{\E N_j}
\]
It can be easily seen that $\mathrm{Var}(M_j)\leq C$ for some $C>0$ which is
independent of $j$. Since $\E M_j=0$,
\[
 \Prob\left(\left.
\left|\frac{1}{j^*}\sum_{j=1}^{j^*}M_j\right| >\lambda \right|j^*>
R^\de\eps^{-2\kk}\right)\to 0
\]
as $\eps\to 0$, which gives the necessary estimates for the first row in
\eqref{eq:ProbExitFinal2}. Therefore, it only remains to bound the second row
in  \eqref{eq:ProbExitFinal2}. To this end, it is enough to point out that 
\[
\left|\frac{1}{j^*}\sum_{j=1}^{j^*}M_j\right| \leq\lambda
\]
implies
\[
 n^*\geq \sum_{j=1}^{j^*-1}N_j \geq  (1-\lambda)(j^*-1)\min_j \E N_j.
\]
Therefore, $n^*\gtrsim R^\de \eps^{-2}$. Nevertheless, by hypothesis, $n^*\leq
s\eps^{-2}$. Therefore, taking $R$ large enough (depending on $s$), we obtain
\[
 \Prob\left(X\left|j^*\leq
R^\de\eps^{-2\kk},\left. \left|\frac{1}{j^*}\sum_{j=1}^{j^*}M_j\right|
\leq\lambda\right.\right.\right)=0.
\]
This completes the proof of the fact that the second row in
\eqref{def:SplittingExpect} goes to zero as $\eps\to 0$ and $R\to+\infty$.


Now we prove that the first row in
\eqref{def:SplittingExpect} goes to zero as $\eps\to 0$ for any fixed $R>0$.
Now we proceed as in the proof of Lemma \ref{lemma:expectlemmaBigStrips} and we
model the visits to 
the strips in $[-R,R]$  as a  symmetric  random walk. The number 
of strips is of order $C(R)\eps^{-\kk}$ for some function $C(R)$ independent of
$\eps$.
%

As in the proof of Lemma \ref{lemma:expectlemmaBigStrips}, we modify slightly
the strips
$\I_\kk^j$. Consider endpoints of the strips 
\[
 r_j=A_j \eps^\kk,\quad j\in\ZZ
\]
with some constants $A_j$ independent of $\eps$ satisfying $A_0=0$, $A_1=A>0$ 
and 
$A_j<A_{j+1}$ for $j>0$ (and analogously for negative $j$'s).  We consider the
strips 
\[
 \I_\kk^j=[r_j,r_{j+1}]=[A_j\eps^\kk, A_j\eps^\kk].
\]
To analyze the visits to these strips, we consider the lattice of points 
$\{r_j\}_{j\in\ZZ}\subset\R$ and we treat the  ``visits'' to these
points.  Lemma \ref{lemma:expectlemmaBigStrips} and the analogous lemma for
resonant zones given in \cite{CGK} imply that 
if we start with $r=r_j$ we hit either $r_{j-1}$ or 
$r_{j+1}$  with probability one. We treat this process   as a
random walk 
for $j\in\ZZ$,
\begin{equation}\label{def:randomwalk}
 S_j=\sum_{i=0}^{j-1}Z_i,
\end{equation}
where $Z_i$ are Bernouilli variables taking values $\pm 1$. We choose properly
the constants $A_j>0$ to have $Z_i$ which  are Bernouilli 
variables with $p=1/2$. That is, to have a
classical symmetric random walk. 

\begin{lem}
There exists constants  $J^\pm>0$  and  $\{A_j\}_{j=\lfloor 
J_-\eps^{-\kk}\rfloor}^{\lfloor 
J_+\eps^{-\kk}\rfloor}$ all independent of $\eps$ such that 
\begin{itemize}
 \item Satisfy 
\[
A_j=A_{j-1}+(A_1-A_{0})\exp\left(-\int_0^{r_{j-1}}\frac { 
2b(r)}{\sigma^2(r)}dr\right)+\OO(\eps^{\kk})
\]
\item $\displaystyle [-R,R]\subset\bigcup_{j=\lfloor 
J_-\eps^{-\kk}\rfloor}^{\lfloor 
J_+\eps^{-\kk}\rfloor}[A_j,A_{j+1}]$.
\item The random walk process induced by the map \eqref{eq:NRmap-n} on the 
lattice $\displaystyle\{r_j\}_{j=\lfloor 
J_-\eps^{-\kk}\rfloor}^{\lfloor 
J_+\eps^{-\kk}\rfloor}$ is a symmetric random walk.
\end{itemize}
\end{lem}
The proof of this lemma is analogous to the proof of Lemma
\ref{lemma:randomwalkinter}.

Now we prove the convergence to zero of the first row in 
\eqref{def:SplittingExpect}. In that case we stay in $[-R,R]$ for all time 
$n\leq s\eps^{-2}$ and we can model the whole evolution as a symmetric  random
walk. Define $j^*$ the number of changes of strip until 
reaching $n=\lfloor s\eps^{-2}\rfloor$. We define the 
Markov times  $0=n_0<n_1<n_2 <\dots <n_{j^*-1}<n_{j^*}<n$
for some random $j^*=j^*(\om)$ such that each $n_j$ is the stopping 
time as in \eqref{eq:stopping-time}. Almost surely $j^*(\om)$ is 
finite.  We decompose the above sum as $\eta_f=\sum_{j=0}^{j^*} \eta_j$ with  
\begin{eqnarray*}
\eta_j=&&f(r_{n_{j+1}})-
f(r_{n_{j}})- \qquad \qquad \\
&&  \eps^2
\sum_{s=n_j}^{n_{j+1}}\left(b(r_s)f'(r_s)+\frac{\sigma^2(r_s)}{2}
f''(r_s)\right).
\end{eqnarray*}
Lemma \ref{lemma:expectlemmaBigStrips} and the analogous lemma for resonant
zones in \cite{CGK} imply that for 
any $j$, 
\begin{equation}\label{def:localExpec}
|\E(\eta_j)|\lesssim\eps^{2\kk+\zeta}
\end{equation}
for some $\zeta>0$. Define $\Delta_j=n_{j+1}-n_j$ and . We split 
$\E(\eta_f)$ as 
\begin{equation}\label{def:SplitGlobalExpec}
 \begin{split}
\E(\eta_f)= & \E\left(\left.\sum_{j=0}^{j^*}\eta_j\right| 
\eps^{-2(1-\kk)+\de}\leq \Delta_j\leq 
\eps^{-2(1-\kk)-\de}\,\,\forall j\right)\\
&\qquad\times\Prob\left(\eps^{-2(1-\kk)+\de}\leq 
\Delta_j\leq 
\eps^{-2(1-\kk)-\de}\,\,\forall j\right)\\
&+\E\left(\left.\sum_{j=0}^{j^*}\eta_j\right| \,\,\exists j \,\text{ s.
t. }
\Delta_j<\eps^{-2(1-\kk)+\delta} \text{ or } 
\Delta_j>\eps^{-2(1-\kk)-\delta}\right)\\
&\qquad\times \Prob\left(\exists 
j \,\text{ s. t. }
\Delta_j<\eps^{-2(1-\kk)+\delta} \text{ or } 
\Delta_j>\eps^{-2(1-\kk)-\delta}\right)
\end{split}
\end{equation}
where $j$ satisfies $0\leq j\leq j^*-1$.

We first bound the second term in the sum. We need 
to estimate how many strips the iterates may 
visit for $n\le s\eps^{-2}$. Proceeding as in the proof of Lemma
\ref{lemma:expectlemmaBigStrips}, since we have $|r_n-r_{n-1}|\lesssim\eps$,
there exists a 
constant $c>0$ such that 
\[
 |n_{j+1}-n_j|\geq c\eps^{\kk-1}\quad\text{ for }j=0,\ldots, j^*-1.
\]
Therefore
\begin{equation}\label{def:maxvisitedstrips:final}
j^*\lesssim \eps^{1-\kk}.
\end{equation}
Then, by  Lemmas \ref{lemma:expectlemmaBigStrips}
and the corresponding lemma for resonant zones in \cite{CGK}, for any small
$\de$,
\[
 \Prob\left(\exists 
k \,\text{ s. t. }
\Delta_j<\eps^{-2(1-\kk)+\delta} \text{ or } 
\Delta_j>\eps^{-2(1-\kk)-\delta}\right)\leq 
\eps^{-1-\kk}e^{-C\eps^{-\de}}.
\]
This implies, 
\[
\begin{split}
\Bigg| \E\Bigg(\sum_{j=0}^{j^*}\eta_j\Bigg| &\,\,\exists j \,\text{ s.
t. }
\Delta_j<\eps^{-2(1-\kk)+\delta} \text{ or } 
\Delta_j>\eps^{-2(1-\kk)-\delta}\Bigg)\Bigg|\\
&\times \Prob\left(\exists 
j \,\text{ s. t. }
\Delta_j<\eps^{-2(1-\kk)+\delta} \text{ or } 
\Delta_j>\eps^{-2(1-\kk)-\delta}\right)\\
&\leq 
\eps^{-1-\kk}\cdot \eps^{2\kk+d}\cdot\eps^{-1-\kk}e^{
-C\eps^{-\de}}.
\end{split}
\]
Now we bound the first term in \eqref{def:SplitGlobalExpec}. Taking into 
account the assumptions on the exit times $\Delta_j$, we can  assume 
\begin{equation}\label{eq:j*:final}
 \eps^{-2\kk+\de}\leq j^*\leq \eps^{-2\kk-\de}.
\end{equation}
Now we are ready to prove that the first term in \eqref{def:SplitGlobalExpec} 
tends to zero with $\eps$. We bound the probability by one. To prove
that 
the conditioned expectation in the first line tends to zero with $\eps$,
it is enough to take into account \eqref{def:localExpec} and 
\eqref{eq:j*:final}, to obtain 
\[
\left|\E\left(\left.\sum_{j=0}^{j^*}\eta_j\right| 
\eps^{-2(1-\kk)+\de}\leq \Delta_j\leq 
\eps^{-2(1-\kk)-\de}\,\,\forall j\right)\right|
\lesssim
\eps^{2\kk+\zeta}\cdot\eps^{-2\kk-\de}
\leq \eps^{\zeta-\de}.
\]
Therefore, taking $\de>0$ small enough we have that the first row in  
\eqref{def:SplittingExpect} tends to zero with $\eps$. This completes the proof 
of  \eqref{def:ExpNF} and therefore of Theorem \ref{maintheorem}.

\appendix

\section{Measure of the domain covered by IR intervals}\label{sec:measure-IR-RR}
A point belongs to a Imaginary Rational strip if it is $\eps^\nu$-close to a
rational number $p/q$ with $|q|\leq \eps^{-b}$ (see \eqref{conditionbeta}).
In this section we show that, with the right choice of $b$, the measure of the
the union of all Imaginary Rational strips  inside any compact set, 
\[
A_{\nu,\ga}=\cup_k I_\ga^k\subset\mathbb{T}\times B\qquad I_\ga^k\ \textrm{
totally irrational}
\]
goes to zero as $\eps\to 0$.  


We do the proof for $A=[0,1]$. The general case is completely analogous. 
Let us consider:
\begin{equation}\nonumber
\mathcal{R}=\left\{p/q\in\mathbb{Q}\,:\,p<q,\,\gcd(p,q)=1,\,q<\eps^{-b}\right\}
=\cup_ { q=1 } ^{q_\textrm{max}}\mathcal{R}_q \subset[0,1],
\end{equation}
where $q_{\textrm{max}}=[\eps^{-b}]$ and:
\begin{equation}\nonumber
 \mathcal{R}_q=\{p/q\in\mathbb{Q}\,:\,p<q,\,\gcd(p,q)=1\}.
\end{equation}
Finally we denote:
\begin{equation}\nonumber
 I_\mathcal{R}= \bigcup_{p/q\in\mathcal{R} }
\left[\frac{p}{q}-2\eps^\nu,\frac{p}{q}+2\eps^\nu \right]\bigcap [0,1].
\end{equation}

\begin{lem}\label{lem:im-rational}
Let $\rho$ be fixed, $0<\rho<\nu$, and define $b=(\nu-\rho)/2$. 
Then, 
\begin{enumerate}
 \item In each  $I_\ga$  there is at most one rational 
$p/q$ in its $\eps^\nu$ neighborhood satisfying $|q|\leq\eps^{-b}$.
\item The Lebesgue measure $\mu$ of the union $I_{\mathcal
R}$ satisfies  $\mu(I_\mathcal{R}) \le \eps^\rho$
and, therefore,  as $\eps \to 0$,
\[
\mu(I_\mathcal{R})\to 0.
\]
\end{enumerate}
\end{lem}
\begin{proof}
On the one hand, suppose that $p/q\in [0,1]$, $q\leq\eps^{-b}$. Then, for all
$p'/q'\in [p/q-\eps^\nu,p/q+\eps^\nu]$, with $p'$ and $q'$ relatively prime and
$p'/q'\neq p/q$, we have
$$\eps^\nu\geq|p/q-p'/q'|\geq\frac{1}{qq'}\geq\frac{\eps^b}{q'}.$$
Therefore, since $b=(\nu-\rho)/2$, 
\[q'\geq\eps^{-\nu+b}=\eps^{-b-\rho}>\eps^{-b},\]
so the first part of the claim is proved.

On the other hand we note that, if $q_1\neq q_2$, then $\mathcal{R}_{q_1}\cap\mathcal{R}_{q_2}=\emptyset$. Moreover, it is clear that $\#\mathcal{R}_q\leq q-1$ (and if $q$ is prime then $\#\mathcal{R}_q=q-1$, so that the bound is optimal). Therefore we have:
$$\#\mathcal{R}\leq\sum_{q=1}^{q_\textrm{max}}\#\mathcal{R}_q\leq\sum_{q=1}^{q_\textrm{max}}q-1=\frac{q_\textrm{max}^2}{2}<\eps^{-2b}.$$
Since $\mu([p/q-\eps^\nu,p/q+\eps^\nu])=\eps^{\nu}$, one has
$$0\leq\mu(I_\mathcal{R})=\eps^{\nu}\#\mathcal{R}<\eps^\nu\eps^{-2b}=\eps^{
\rho },$$ 
which proves the  second claim of the lemma.
\end{proof}

\section{An auxiliary lemma}
\label{sec:auxiliaries}

To estimate the exit time, we need the following 
auxiliary lemma. Consider the random sum
\begin{align}\label{random sum}
S_n=\sum_{k=1}^n v_k\om_k,\ \ n\ge 1,
\end{align}
where $\{\om_k\}_{k\ge 1}$ is a sequence of 
independent random variables with equal $\pm 1$ with 
equal probability $1/2$ each and $\{v_k\}_{k\ge 1}$ is 
a sequence such that 
\[
\lim_{n\to \infty} \frac{\sum_{k=1}^n v_k^2}{n}=\sigma.
\] 

\begin{lem}\label{main lemma}
$\{S_n/n^{1/2}\}_{n\ge 1}$ converges in distribution to 
the normal distribution $\mathcal N(0,\sigma^2)$.
\end{lem}
\begin{proof}

Recall that a characteristic function of a random 
variable $X$ is a function $\phi_X:\R \to \mathbb C$ given by 
$\phi_X(t)=\E \exp (itX)$. Notice that it satisfies the 
following two properties:
\begin{itemize}
\item If $X,Y$ are independent random variables, 
then $\varphi_{X+Y}=\varphi_X\cdot\varphi_Y$.
\item $\varphi_{aX}(t)=\varphi_X(at)$.
\end{itemize}

A sufficient condition to prove convergence in distribution
is as follows.  
\begin{thm}[Continuity theorem \cite{Br}]
\label{continuity theorem}
Let $\{X_n\}_{n\ge 1},Y$ be random variables. 
If $\{\varphi_{X_n}(t)\}_{n\ge 1}$ converges to $\varphi_Y(t)$
for every $t\in\mathbb R$, then $\{X_n\}_{n\ge 1}$ converges 
in distribution to $Y$.
\end{thm}

A direct calculation shows that 
\[
\lim_{n\to \infty} \log \phi_{S_n/\sqrt n}(t)=-\dfrac{\sigma^2t^2}{2}
\qquad \text{ for all }\ \  t\in \R.   
\] 
This way of proof was communicated to the authors by 
Yuri Lima. 
\end{proof}

{\bf Acknowledgement:} The authors warmly thank Leonid Koralov 
for numerious envigorating discussions of various topics involving 
stochatic processes. Communications with Dmitry Dolgopyat, 
Yuri Bakhtin, Jinxin Xue were useful for the project 
and gladly acknowledged by the authors. The first and second authors have 
been partially supported by the  Spanish
MINECO-FEDER Grant MTM2015-65715 and the Catalan Grant 2014SGR504. The third 
author  acknowledges partial support of the NSF grant DMS-1402164.

\end{document}